\documentclass[11pt,leqno]{article}
\usepackage[margin=1in]{geometry} 
\usepackage{amssymb,amsfonts,amsmath,bbm,mathrsfs,stmaryrd,mathtools}
\usepackage{xcolor}
\usepackage{url}

\usepackage{extarrows}

\usepackage[shortlabels]{enumitem}
\usepackage{tensor}

\usepackage{xr}
\usepackage[T1]{fontenc}

\usepackage[colorlinks,
linkcolor=black!75!red,
citecolor=blue,
pdftitle={},
pdfproducer={pdfLaTeX},
pdfpagemode=None,
bookmarksopen=true,
bookmarksnumbered=true,
backref=page]{hyperref}

\usepackage{tikz}
\usetikzlibrary{arrows,calc,decorations.pathreplacing,decorations.markings,intersections,shapes.geometric,through,fit,shapes.symbols,positioning,decorations.pathmorphing}

\usepackage{braket}

\usepackage[amsmath,thmmarks,hyperref]{ntheorem}
\usepackage{cleveref}

\newcommand\nthalias[1]{\AddToHook{env/#1/begin}{\crefalias{lemma}{#1}}}

\nthalias{definition}
\nthalias{example}
\nthalias{examples}
\nthalias{remark}
\nthalias{remarks}
\nthalias{convention}
\nthalias{notation}
\nthalias{construction}
\nthalias{theoremN}
\nthalias{propositionN}
\nthalias{corollaryN}
\nthalias{lemma}
\nthalias{proposition}
\nthalias{corollary}
\nthalias{theorem}
\nthalias{conjecture}
\nthalias{question}
\nthalias{assumption}

\creflabelformat{enumi}{#2#1#3}

\crefname{section}{Section}{Sections}
\crefformat{section}{#2Section~#1#3} 
\Crefformat{section}{#2Section~#1#3} 

\crefname{subsection}{\S}{\S\S}
\AtBeginDocument{%
  \crefformat{subsection}{#2\S#1#3}%
  \Crefformat{subsection}{#2\S#1#3}%
}

\crefname{subsubsection}{\S}{\S\S}
\AtBeginDocument{%
  \crefformat{subsubsection}{#2\S#1#3}%
  \Crefformat{subsubsection}{#2\S#1#3}%
}

\theoremstyle{plain}

\newtheorem{lemma}{Lemma}[section]
\newtheorem{proposition}[lemma]{Proposition}
\newtheorem{corollary}[lemma]{Corollary}
\newtheorem{theorem}[lemma]{Theorem}

\theoremstyle{plain}
\theoremnumbering{Alph}
\newtheorem{theoremN}{Theorem}

\theoremstyle{plain}
\theorembodyfont{\upshape}
\theoremsymbol{\ensuremath{\blacklozenge}}

\newtheorem{definition}[lemma]{Definition}
\newtheorem{example}[lemma]{Example}

\newtheorem{remark}[lemma]{Remark}
\newtheorem{remarks}[lemma]{Remarks}

\newtheorem{notation}[lemma]{Notation}

\crefname{definition}{definition}{definitions}
\crefformat{definition}{#2definition~#1#3} 
\Crefformat{definition}{#2Definition~#1#3} 

\crefname{ex}{example}{examples}
\crefformat{example}{#2example~#1#3} 
\Crefformat{example}{#2Example~#1#3} 

\crefname{exs}{example}{examples}
\crefformat{examples}{#2example~#1#3} 
\Crefformat{examples}{#2Example~#1#3} 

\crefname{remark}{remark}{remarks}
\crefformat{remark}{#2remark~#1#3} 
\Crefformat{remark}{#2Remark~#1#3} 

\crefname{remarks}{remark}{remarks}
\crefformat{remarks}{#2remark~#1#3} 
\Crefformat{remarks}{#2Remark~#1#3} 

\crefname{convention}{convention}{conventions}
\crefformat{convention}{#2convention~#1#3} 
\Crefformat{convention}{#2Convention~#1#3} 

\crefname{notation}{notation}{notations}
\crefformat{notation}{#2notation~#1#3} 
\Crefformat{notation}{#2Notation~#1#3} 

\crefname{table}{table}{tables}
\crefformat{table}{#2table~#1#3} 
\Crefformat{table}{#2Table~#1#3}

\crefname{lemma}{lemma}{lemmas}
\crefformat{lemma}{#2lemma~#1#3} 
\Crefformat{lemma}{#2Lemma~#1#3} 

\crefname{proposition}{proposition}{propositions}
\crefformat{proposition}{#2proposition~#1#3} 
\Crefformat{proposition}{#2Proposition~#1#3} 

\crefname{propositionN}{proposition}{propositions}
\crefformat{propositionN}{#2proposition~#1#3} 
\Crefformat{propositionN}{#2Proposition~#1#3} 

\crefname{corollary}{corollary}{corollaries}
\crefformat{corollary}{#2corollary~#1#3} 
\Crefformat{corollary}{#2Corollary~#1#3} 

\crefname{corollaryN}{corollary}{corollaries}
\crefformat{corollaryN}{#2corollary~#1#3} 
\Crefformat{corollaryN}{#2Corollary~#1#3} 

\crefname{theorem}{theorem}{theorems}
\crefformat{theorem}{#2theorem~#1#3} 
\Crefformat{theorem}{#2Theorem~#1#3} 

\crefname{theoremN}{theorem}{theorems}
\crefformat{theoremN}{#2theorem~#1#3} 
\Crefformat{theoremN}{#2Theorem~#1#3} 

\crefname{enumi}{}{}
\crefformat{enumi}{#2#1#3}
\Crefformat{enumi}{#2#1#3}

\crefname{assumption}{assumption}{Assumptions}
\crefformat{assumption}{#2assumption~#1#3} 
\Crefformat{assumption}{#2Assumption~#1#3} 

\crefname{construction}{construction}{Constructions}
\crefformat{construction}{#2construction~#1#3} 
\Crefformat{construction}{#2Construction~#1#3} 

\crefname{question}{question}{Questions}
\crefformat{question}{#2question~#1#3} 
\Crefformat{question}{#2Question~#1#3} 

\crefname{equation}{}{}
\crefformat{equation}{(#2#1#3)} 
\Crefformat{equation}{(#2#1#3)}

\numberwithin{equation}{section}
\renewcommand{\theequation}{\thesection-\arabic{equation}}

\theoremstyle{nonumberplain}
\theoremsymbol{\ensuremath{\blacksquare}}

\newtheorem{proof}{Proof}
\newcommand\pf[1]{\newtheorem{#1}{Proof of \Cref{#1}}}

\newcommand\bC{{\mathbb C}}
\newcommand\bD{{\mathbb D}}

\newcommand\bP{{\mathbb P}}
\newcommand\bQ{{\mathbb Q}}
\newcommand\bR{{\mathbb R}}
\newcommand\bS{{\mathbb S}}
\newcommand\bT{{\mathbb T}}

\newcommand\bZ{{\mathbb Z}}

\newcommand\cA{{\mathcal A}}

\newcommand\cE{{\mathcal E}}
\newcommand\cF{{\mathcal F}}

\newcommand\fm{{\mathfrak m}}

\DeclareMathOperator{\lcm}{lcm}

\DeclareMathOperator{\End}{\mathrm{End}}

\DeclareMathOperator{\Max}{\mathrm{Max}}

\newcommand\numberthis{\addtocounter{equation}{1}\tag{\theequation}}

\newcommand{\cat}[1]{\textsc{#1}}

\newcommand{\qedhere}{\mbox{}\hfill\ensuremath{\blacksquare}}

\newcommand{\xrightarrowdbl}[2][]{%
  \xrightarrow[#1]{#2}\mathrel{\mkern-14mu}\rightarrow
}

\title{Polynomial identities and Azumaya loci for rational quantum spheres}
\author{Alexandru Chirvasitu}

\begin{document}

\date{}

\newcommand{\Addresses}{{
  \bigskip
  \footnotesize

  \textsc{Department of Mathematics, University at Buffalo}
  \par\nopagebreak
  \textsc{Buffalo, NY 14260-2900, USA}  
  \par\nopagebreak
  \textit{E-mail address}: \texttt{achirvas@buffalo.edu}


}}

\maketitle

\begin{abstract}
  We prove a number of structure and isomorphism results concerning the non-commutative Natsume-Olsen spheres $\mathbb{S}^{2n-1}_{\theta}$ deformed along a skew-symmetric matrix $\theta\in \mathbb{R}$. These include (a) the fact that two $C^*$-algebras of the form $\mathbb{S}^{3}_{\theta}\otimes M_n$ are isomorphic precisely in the obvious cases; (b) the fact that $m$ and $n$ are recoverable from the isomorphism class of $C(\mathbb{S}^{2m-1}_{\theta})\otimes M_n$; (c) the PI character, PI degree and Azumaya loci of $C(\mathbb{S}^{2m-1}_{\theta})$ for rational $\theta$, along with a realization of their centers as (function algebras of) branched cover of $\mathbb{S}^{2n-1}$ and (d) for rational $\theta$ again, the topological finite generation of $C(\mathbb{S}^{2m-1}_{\theta})$ over their centers, with algebraic finite generation equivalent to being classical (equivalently, Azumaya).
\end{abstract}


\vspace{.5cm}

\noindent{MSC 2020: 46L52; 16H05; 46M20; 55R25; 55R37; 46L85; 16R10; 55R40


}


\section*{Introduction}

Consider a skew-symmetric real matrix $\theta\in M_n(\bR)$. We will be working extensively with the following non-commutative-geometric constructs.
\begin{itemize}[wide]
\item The \emph{non-commutative} (or \emph{quantum}) \emph{tori} $\bT^n_{\theta}$ (\cite[\S 1]{rief_case}, \cite[\S 12.2]{gvf_ncg}, \cite[\S 1.1.5]{khal_basic}, etc.) defined as objects dual to the generator-and-relation $C^*$-algebras
  \begin{equation*}
    A_{\theta}^n
    =
    C\left(\bT^n_{\theta}\right)
    :=
    \Braket{\text{unitaries $u_j$},\ j\in [n]:=\{1..n\}\ \mid\ u_k u_j = e(\theta_{jk})u_j u_k}    
  \end{equation*}
  for $e(-):=\exp(2\pi i -)$;

\item and the analogously-defined \emph{non-commutative spheres} \cite[Definition 2.1]{no_sph} defined by
  \begin{equation*}
    C\left(\bS_{\theta}^{2n-1}\right)
    :=
    \Braket{\text{normal }t_j
      ,\ j\in [n]
      \ \bigg|\
      \begin{aligned}
        t_k t_j &= e(\theta_{jk})t_j t_k\\
        \sum t_j^* t_j &= 1
      \end{aligned}
      }.
  \end{equation*}
\end{itemize}
The former ``glue'' to produce the latter in ways reminiscent of classical topology (e.g. the \emph{standard genus-1 (Heegaard) splitting} $\bS^3=\left(\bT^2\times [0,1]\right)\cup_{\bT^2}\left(\bT^2\times [0,1]\right)$ of \cite[Proposition 3.3]{MR1886684}): per \cite[Theorem 2.5]{no_sph}, we have
\begin{equation}\label{eq:sph2tor}
  C\left(\bS^{2n-1}_{\theta}\right)
  \cong
  \mathrm{Cont}_{\partial}
  \left(
    \bS_+^{n-1}\xrightarrow{\quad}C\left(\bT^n_{\theta}\right)
  \right),
\end{equation}
where
\begin{equation}\label{eq:bdry.cond}
  \begin{aligned}
    \bS_+^{n-1}
    &:=
      \left\{(s_1,\ \cdots,\ s_n)\in \bS^{n-1}\subseteq \bR^n\ :\ s_i\ge 0\right\}\\
    \bS^{n-1}_{+,F\subseteq [n]}
    &:=
      \left\{(s_i)\in \bS_+^{n-1}\ :\ s_i=0,\ \forall i\not\in F\right\}
      \quad\text{and `$\partial$' means}\\
    f\left(\bS^{n-1}_{+,F}\right)
    &\subseteq
      C^*\left(u_i,\ i\in F\right)
      \subseteq A_{\theta}^n
      ,\quad\forall F\subseteq [n].
  \end{aligned}
\end{equation}

\Cref{se:iso.pb} is concerned with reconstruction/isomorphism problems, i.e. the extent to which initial data (the deformation parameter $\theta\in M_n(\bR)$, or perhaps its size $n$) are determined by the isomorphism or stable isomorphism classes of the quantum-sphere algebras. A paraphrased aggregate of \Cref{th:theta3sphiso,th:recovermatsize} reads as follows. 

\begin{theoremN}\label{thn:iso.cls}
  \begin{enumerate}[(1),wide]
  \item For $n,n'\in \bZ_{>0}$ and $\theta,\theta'\in \bR$ identified with skew-symmetric $2\times 2$ matrices we have
    \begin{equation*}
      \begin{aligned}
        C(\bS^3_{\theta})\otimes M_n
        \cong
        C(\bS^3_{\theta'})\otimes M_{n'}
        &\iff
          C(\bT^2_{\theta})\otimes M_n
          \cong
          C(\bT^2_{\theta'})\otimes M_{n'}\\
        &\iff
          n=n'
          \quad\text{and}\quad
          \theta\in \pm\theta'+\bZ.
      \end{aligned}      
    \end{equation*}

  \item For positive integers $m,n$ with $m\ge 2$ and skew-symmetric $\theta\in M_m(\bR)$ the isomorphism class of either $C(\bT^{m}_{\theta})\otimes M_n$ or $C(\bS^{2m-1}_{\theta})\otimes M_n$ determines $m$ and $n$.  \qedhere
  \end{enumerate}
\end{theoremN}

In \Cref{se:high.dim} and onward we assume the deformation parameter $\theta\in M_n(\bQ)$ rational and examine consequent \emph{polynomial-identity (PI)} phenomena (on which background the main text offers a brief reminder). The focus is on
\begin{itemize}[wide]
\item \emph{Azumaya algebras} \cite[\S 4.4]{dcp_qg}: one way to formalize a purely algebraic analogue of the section-space $\Gamma\left(\cE\otimes \cE^*\right)$ for a vector bundle $\cE$ over a compact Hausdorff space;

\item and the extent to which algebras fail to qualify, as measured by the \emph{Azumaya locus} \cite[\S III.1.7]{bg_algqg}. 
\end{itemize}
Again summarizing for brevity, one rendition of \Cref{th:s2n1thetapi,th:ctheta.azumaya.loc} below is as follows. 

\begin{theoremN}\label{thn:pi}
  Let $\theta\in M_n(\bQ)$ be a skew-symmetric matrix for $n\in \bZ_{\ge 2}$.

  \begin{enumerate}[(1),wide]
  \item The center $Z_{\theta}\le C_{\theta}:=C(\bS_{\theta}^{2n-1})$ is the function algebra of a branched cover of $\bS^{2n-1}$ which is not, in general, a topological manifold.

  \item The smallest $n^2$ admitting a $Z_{\theta}$-algebra embedding $C_{\theta}\le M_n(Z_{\theta})$ is
    \begin{equation*}
      h_{\theta}
      :=
      [(\bZ^n+\mathrm{im}~\theta)\ :\ \bZ^n],
    \end{equation*}
    so in particular $C_{\theta}$ is a PI algebra of \emph{PI-degree} \cite[\S 1]{zbMATH04101395} $\sqrt{h_{\theta}}$.

  \item The Azumaya locus of $C_{\theta}$ consists precisely of those maximal ideals $p\in \Max(Z_{\theta})$ whose restriction to $\bS_+^{n-1}$ along \Cref{eq:sph2tor} belongs to some face supported by $F\subseteq [n]$ with
    \begin{equation*}
      h_{\theta}
      >
      h_{\theta,F}
      :=
      \text{$h$ attached to the sub-matrix of $\theta$ supported on $F$-indexed rows/columns}.
    \end{equation*}
    \qedhere
  \end{enumerate}
\end{theoremN}

In \Cref{se:cent.fin} we turn to the intimate connection between polynomial identities and the condition that an algebra be finitely-generated as a module over its center (see e.g. \cite[Chapter 6]{proc_pi}, the motivating discussion on \cite[p.532]{zbMATH03351818}, etc.) in the context of studying the quantum spheres $\bS^{2n-1}_{\theta}$. It will turn out that said finite generation (mostly) fails, but its weaker topological analogue (requiring that $C\left(\bS^{2n-1}_{\theta}\right)$ contain a dense finitely-generated module over its center) always holds; per \Cref{th:ct.tfg.not.fg}:

\begin{theoremN}\label{thn:cent.fin}
  For $n\in \bZ_{\ge 2}$ and rational skew-symmetric $\theta\in M_n(\bQ)$ the algebra $C\left(\bS_{\theta}^{2n-1}\right)$ is
  \begin{itemize}[wide]
  \item always a topologically finitely-generated module over its center $Z_{\theta}$;

  \item but \emph{algebraically} finitely-generated as such precisely when $\theta$ is integral.  \qedhere
  \end{itemize}
\end{theoremN}

\Cref{thn:pi,thn:cent.fin} both link naturally to the theory of \emph{Banach, Hilbert} and \emph{$C^*$ bundles} (\cite[\S II.13]{fd_bdl-1}, \cite[pp.7-9]{dg_ban-bdl}, \cite[\S 1]{dupr_clsf-cast-bdl}, etc.) over compact Hausdorff spaces and satellite topics: the theory of \emph{non-commutative branched covers} initiated in \cite{pt_brnch} and phrased in the language of \emph{finite-index conditional expectations} \cite[Definition 2]{fk_fin-ind} is germane to the discussion below, which relies directly or indirectly on material from \cite{bg_cx-exp,2409.03531v1,pt_brnch}.

\subsection*{Acknowledgments}

The paper has benefited from fruitful exchanges with B. Badzioch, M. Khalkhali, B. Passer and I. Thompson.


\section{Isomorphisms of rationally-deformed quantum 3-spheres}\label{se:iso.pb}

Throughout, unqualified (typically vector or algebra) \emph{bundles} are assumed \emph{locally trivial} \cite[Definition 1.1.8]{hjjm_bdle}; they are to be distinguished from more general constructs termed \emph{(F) Banach bundles} on \cite[pp.7-8]{dg_ban-bdl}, which will also make an appearance in \Cref{se:cent.fin}.

The isomorphism problem for non-commutative 3-spheres is not difficult to resolve, given its 2-torus analogue. The following is very much in the spirit of \cite[Theorem 3]{rief_irrat} for {\it irrational} tori. The {\it rational} torus version \cite[Theorem 1.1]{hks_erg} (recovered also as \cite[Theorem 3.12]{rief_canc}) typically does not involve matrix tensorands.

\begin{theorem}\label{th:theta3sphiso}
  Consider $n,n'\in \bZ_{>0}$,  $\theta,\theta'\in \bR$, and set
  \begin{equation*}
    C_{\theta,n}:=C(\bS^3_{\theta})\otimes M_n
    ,\quad
    A_{\theta,n}:=C(\bT^2_{\theta})\otimes M_n,
  \end{equation*}
  and similarly for the primed parameters. The following conditions are equivalent.
  \begin{enumerate}[(a),wide]
  \item\label{item:3sphiso} We have an isomorphism
    \begin{equation}\label{eq:ctnctn}
      C_{\theta,n}
      \cong
      C_{\theta',n'}.
    \end{equation}

  \item\label{item:2toriso} We have an isomorphism
    \begin{equation}\label{eq:atnatn}
      A_{\theta,n}
      \cong
      A_{\theta',n'}.
    \end{equation}

  \item\label{item:samepara} The parameters coincide save for trivial modifications, in the sense that
    \begin{equation*}
      n=n'
      \quad\text{and}\quad
      \theta\in \pm\theta'+\bZ.
    \end{equation*}
  \end{enumerate}
\end{theorem}


It might be worthwhile to first observe separately that the size $n$ of the tensorand $M_n$ can be recovered from the isomorphism class of either $C_{\theta,n}$ or $A_{\theta,n}$ and in fact, more generally, for non-commutative spheres/tori of arbitrary dimension.

\begin{theorem}\label{th:recovermatsize}
  Let $\theta\in M_m(\bR)$ be a skew-symmetric matrix, $m\ge 2$ and $n\in \bZ_{>0}$.

  \begin{enumerate}[(1),wide]

  \item\label{item:torirecn} The isomorphism class of the $C^*$-algebra $C(\bT^{m}_{\theta})\otimes M_n$ determines $m$ and $n$.

  \item\label{item:sphrecn} The isomorphism class of the $C^*$-algebra $C(\bS^{2m-1}_{\theta})\otimes M_n$ determines $m$ and $n$. 

    
  \end{enumerate}
\end{theorem}
\begin{proof}
  We structure the argument conforming to the statement.

  \begin{enumerate}[label={},wide=0pt]

  \item \textbf{ \Cref{item:torirecn}} This is a simple matter of unwinding some of the well-known K-theoretic results on non-commutative tori.

    According to \cite[Theorem 2.2]{el_kproj} $K_0(C(\bT^m_{\theta}))$ is free abelian of rank $2^{m-1}$ with the class of the identity as a generator. The usual isomorphism
    \begin{equation*}
      K_0(A)\cong K_0(A\otimes M_n)
    \end{equation*}
    induced by the upper-corner embedding \cite[Lemma 6.2.10]{wo} then makes it clear that in
    \begin{equation*}
      K_0\left(C(\bT^{m}_{\theta})\otimes M_n\right)\cong \bZ^{2^{m-1}}
    \end{equation*}
    the class of the identity is divisible precisely by $n$, hence the conclusion.
    
  \item \textbf{ \Cref{item:sphrecn}: recovering $n$.} \Cref{eq:sph2tor} realizes $C(\bS_{\theta}^{2m-1})$ as a \emph{$C(X)$-algebra} (i.e. \cite[Definition 1.5]{zbMATH04056334} a $C^*$-algebra receiving a central morphism from $C(X)$) for $X=\bS_+^{n-1}$. The \emph{fibers}
    \begin{equation*}
      C(\bS_{\theta}^{2m-1})|_p
      :=
      C(\bS_{\theta}^{2m-1})
      /
      \left(C(\bS_{\theta}^{2m-1})\cdot \left\{f\in C(X)\ :\ f(p)=0\right\}\right)
      ,\quad
      p\in \bS_+^{n-1}
    \end{equation*}
    are non-commutative-torus algebras $A$ in general, and specifically to $C(\bS^1)$ at the $n$ vertices of the spherical simplex $X$. It follows that $n$ can be recovered from the abstract $C^*$-algebra $C:=C(\bS^{2m-1}_{\theta})\otimes M_n$ as the smallest dimension of an irreducible representation.

  \item \textbf{ \Cref{item:sphrecn}: recovering $m$.} A small detour will first discern what can be recovered from the center 
    \begin{equation}\label{eq:zalone}
      Z:=Z\left(C(\bS^{2m-1}_{\theta})\otimes M_n\right)
      =
      Z\left(C(\bS^{2m-1}_{\theta})\right)
      =
      Z(C)
    \end{equation}
    alone. The isomorphism \Cref{eq:sph2tor} specializes to $Z \cong \mathrm{Cont}_{\partial}\left(\bS_+^{m-1}\xrightarrow{\quad}Z\left(C(\bT^n_{\theta})\right)\right)$, with `$\partial$' indicating the central analogue
    \begin{equation}\label{eq:centbdr}
      f\left(\bS^{m-1}_{+,F}\right)
      \subseteq
      Z\left(C(\bT^n_{\theta})\right)\cap C^*(\{u_j\ |\ j \in F\}).
    \end{equation}
    of \Cref{eq:bdry.cond}. 
    
    Recall \cite[\S 2.2]{el_morita-def} that $C(\bT^n_{\theta})$ is a cocycle twist $C^*(\bZ^n,\sigma)$ of the group algebra $C^*(\bZ^n)$ (with the generators of $\bZ^n$ mapping to the $u_i$). The proof of \cite[Lemma 2.3]{el_kproj} then describes the center $Z\left(C(\bT^n_{\theta})\right)$ as 
    \begin{equation}\label{eq:gammaker}
      C^*(\Gamma)\subset C^*(\bZ^n)
      ,\quad
      \Gamma:=\text{kernel of the bicharacter }
      \bZ^n\wedge \bZ^n\xrightarrow{e(\theta) = \exp(2\pi i \theta)} \bS^1
    \end{equation}
    with the matrix $\theta$ regarded, as usual, as a bilinear form. The boundary condition \Cref{eq:centbdr} then reads
    \begin{equation*}
      f\left(\bS^{m-1}_{+,F}\right)
      \subseteq
      C^*(\Gamma_F)\cong C(\bT_F),\ \bT_F:=\widehat{\Gamma_F},
    \end{equation*}
    where
    \begin{equation}\label{eq:zf}
      \Gamma_F:=\Gamma\cap \bZ^F,\quad \bZ^F:=\text{sum of the $F$-indexed summands in }\bZ^n.
    \end{equation}
    We write $\Max(A)$ for the \emph{maximal spectrum} \cite[Exercise 1.26]{am_comm} of a commutative ring, typically applying the notion to commutative $C^*$-algebras (and occasionally omitting the qualifier ``maximal''). $X:=\Max(Z)$ can thus be described as follows:
    \begin{itemize}
    \item To each finite-set inclusion
      \begin{equation*}
        F\subseteq F'\subseteq [m]
      \end{equation*}
      associate
      \begin{equation*}
        \begin{aligned}
          \text{an inclusion}\quad \Delta_F&\lhook\joinrel\xrightarrow{\quad} \Delta_{F'},\quad \Delta_F:=\bS^{m-1}_{+,F}\\
          \text{and a quotient}\quad \bT_F&\xrightarrowdbl{\quad}\bT_{F'}\quad \text{dual to}\quad\Gamma_F\le \Gamma_{F'}.
        \end{aligned}
      \end{equation*}

    \item And then recover $X$ as the {\it coend} \cite[\S IX.6]{mcl_work}
      \begin{equation}\label{eq:xiscoend}
        X=\Max\left(Z\left(C\left(\bS_{\theta}^{2m-1}\right)\right)\right)\cong \int^F \Delta_F\times \bT_F.
      \end{equation}
    \end{itemize}
    This is a particular instance of the {\it geometric realization} \cite[Definition 3.8.1]{riehl_catht} of a {\it simplicial object} \cite[\S VII.5]{mcl_work}: more precisely, $F\xmapsto{} \bT_F$ is a simplicial topological space (i.e. a simplicial object in the category of topological spaces), and \Cref{eq:xiscoend} is its simplicial realization.

    The space \Cref{eq:xiscoend} does have {\it dimension}
    \begin{equation}\label{eq:dimx}
      \begin{aligned}
        \dim X &= m-1+\dim \Max(Z(C(\bT^m_{\theta})))\\
               &= m-1+\mathrm{rank}\ker\theta,
      \end{aligned}      
    \end{equation}
    where
    \begin{itemize}
    \item `dimension' is used in any of three senses: {\it small inductive} \cite[Definition 1.1.1]{engel_dim} or {\it large inductive} \cite[Definition 1.6.1]{engel_dim} or {\it covering} \cite[Definition 1.6.7]{engel_dim} dimension, all coincident \cite[Theorem 1.7.7]{engel_dim} for separable metric spaces;

    \item and the {\it rank} of an abelian group is \cite[Definition A1.59]{hm4}
      \begin{equation*}
        \mathrm{rank}~\Gamma:=\dim_{\bQ}~ \bQ\otimes_{\bZ}\Gamma.
      \end{equation*}
    \end{itemize}
    By contrast to $m$ (\Cref{ex:centernotdetm}), this latter number \Cref{eq:dimx} {\it is}, then, recoverable from the isomorphism class of $Z=Z(C)$ alone (so $C$ itself is not needed for this).

    The fibers of $C$ at the various points of $X$ are of the form $C^*(\bZ^m/\Gamma_F,\ \sigma)\otimes M_n$ for subsets $F\subseteq [m]$ with the twist induced by the original cocycle $\sigma$ on $\bZ^n/\Gamma_F$ (possible, given that $\Gamma=\Gamma_{[m]}$ is by definition the kernel of $\theta$).

    The abelian groups $\bZ^m/\Gamma_F$ have ranks
    \begin{equation}\label{eq:quotranks}
      m-\mathrm{rank}~\Gamma_F
      \ge
      m-\mathrm{rank}~\Gamma_{[m]}
      =
      m-\mathrm{rank}\ker\theta,
    \end{equation}
    with equality achieved. Since the sum between \Cref{eq:dimx} and this minimum is $2m-1$, it will be enough to observe that the ranks on the left-hand side of \Cref{eq:quotranks} can be recovered from the corresponding fibers $C^*(\bZ^m/\Gamma_F,\ \sigma)\otimes M_n$.

    To see this,
    \begin{itemize}
    \item write
      \begin{equation*}
        \bZ^m/\Gamma_F\cong \left(\text{finite abelian group }H\right)\times \bZ^r,\quad r:={\mathrm{rank}~\bZ^m/\Gamma_F}
      \end{equation*}
      \cite[p.VII.19, Theorem 2]{bourb_alg-04-07};

    \item so that $C^*(\bZ^m/\Gamma_F,\ \sigma)$ can be expressed as an iterated, $r$-fold crossed product
      \begin{equation*}
        C^*(H)\rtimes \bZ\rtimes\cdots\rtimes \bZ
      \end{equation*}
      as in \cite[Proposition 12.8]{gvf_ncg};

    \item whence
      \begin{equation*}
        K_0(C^*(\bZ^m/\Gamma_F,\ \sigma)\otimes M_n)
        \cong
        K_0(C^*(\bZ^m/\Gamma_F,\ \sigma))
        \cong \bZ^{2^{r-1}}
      \end{equation*}
      as for non-commutative tori, via the {\it Pimsner-Voiculescu sequence} \cite[Theorem 10.2.1]{blk}. 
    \end{itemize}
  \end{enumerate}
  This completes the proof. 
\end{proof}

It is perhaps worth noting at this point that the isomorphism class of the center \Cref{eq:zalone} alone does {\it not}, in general, determine $m$.

\begin{example}\label{ex:centernotdetm}
  We exhibit non-commutative-sphere algebras $C(\bS_{\theta}^{2m-1})$ with isomorphic centers and distinct $m$ (so also, by necessity, distinct $\theta$):
  \begin{enumerate}[(I),wide=0pt]
  \item Take first $m=2$ and $\theta=0$, so that
    \begin{equation*}
      C(\bS_{\theta}^{2m-1})\cong C(\bS^3)
    \end{equation*}
    and the spectrum of the center is the 3-sphere.

  \item On the other hand, take $m=3$ and some $\theta$ for which the kernel $\Gamma$ of \Cref{eq:gammaker} is of rank 1 (i.e. isomorphic to $\bZ$) and not contained in any of the subgroups
    \begin{equation*}
      \bZ^F,\quad F\subset [3]\text{ properly}
    \end{equation*}
    of \Cref{eq:zf}. The coend \Cref{eq:xiscoend} is then
    \begin{equation*}
      \bS^2_{+}\times \bS^1/\left((x,t)=(x,t'),\ \forall x\in\partial S^2_{+},\ t\in \bS^1\right)
    \end{equation*}
    with `$\partial$' denoting the boundary. Alternatively, in words: consider the product $\bD^2\times \bS^1$ (where $\bD^2$ is the closed 2-disk) and identify all copies
    \begin{equation*}
      \bS^1\times \{t\}\subset \bS^1\times \bS^1\subset \bD^2\times \bS^1
    \end{equation*}
    to each other in the obvious fashion. It is a simple exercise to see that this space is again (homeomorphic to) the 3-sphere.   \end{enumerate}
\end{example}


Also in reference to \Cref{th:recovermatsize}, note that an isomorphism
\begin{equation}\label{eq:ana'n}
  A\otimes M_n\cong A'\otimes M_n
\end{equation}
of unital $C^*$-algebras does not, in general, imply that $A$ and $A'$ themselves are isomorphic:

\begin{example}\label{ex:notcancel}
  Recall \cite[\S 2]{mcc_kext} the unital $C^*$-algebras $A=U^{nc}_{m,n}$ generated by the $m n$ elements of a unitary (possibly non-square) $m\times n$ matrix. As explained in loc.cit., we may as well assume that $m\le n$.
  
  This is (by definition) the universal unital $C^*$-algebra $A$ whose Hilbert modules $A^m$ and $A^n$ are isomorphic, so their endomorphism algebras (as Hilbert modules) will of course also be isomorphic:
  \begin{equation*}
    A\otimes M_m\cong M_m(A)\cong M_n(A)\cong A\otimes M_n.
  \end{equation*}
  Fix, now, any $m\ge 2$, and consider the case $A=U^{nc}_{m,m^2}$. We then have
  \begin{equation*}
    A\otimes M_m\cong A\otimes M_{m^2}\cong \left(A\otimes M_m\right)\otimes M_m,
  \end{equation*}
  i.e. \Cref{eq:ana'n} with $A'=A\otimes M_m$ and $n=m$. An isomorphism
  \begin{equation}\label{eq:aam}
    A\cong A'\cong A\otimes M_m
  \end{equation}
  does not exist though: according to \cite[discussion immediately preceding \S 6]{ag_sepgrph}
  \begin{equation*}
    K_0(A)\cong \bZ/(n-m)=\bZ/(m^2-m)
    \text{ generated by the class of }1\in A,
  \end{equation*}
  so that class cannot be a multiple of $m$ (as it would be, given an isomorphism \Cref{eq:aam}).
\end{example}

The ensuing discussion freely references \emph{Azumaya algebras} (in their various incarnations: \cite[Definition 5.1.4]{agpr_pi}, \cite[\S 13.7.6]{mr}, \cite[\S III.1.3]{bg_algqg}, etc.) and ancillary notions, including the following. 

\begin{itemize}[wide]
\item \emph{polynomial identities} (or \emph{PIs} for short) \cite[Definition 2.2.1]{agpr_pi} for (always unital) rings $A$, i.e. non-commutative polynomials vanishing identities when substituting arbitrary elements of $A$ for the variables;

\item \emph{polynomial identity (PI) algebras} \cite[Definition 2.2.41]{agpr_pi}, i.e. those algebras $A$ over commutative rings (usually, here, fields) $\Bbbk$ satisfying at least one non-zero PI \emph{stable} \cite[\S 2.2.3]{agpr_pi} in the sense that all $A\otimes_{\Bbbk}\Bbbk'$ for commutative-ring extensions $\Bbbk\subseteq \Bbbk'$;

\item An algebra $A$ over a commutative ring $R$ is \emph{Azumaya} (\emph{of rank $n^2$}) as such \cite[Definition 5.1.4]{agpr_pi} if for some faithfully flat extension $R\to R'$ we have
  \begin{equation*}
    A\otimes_R R'\cong M_n(R').
  \end{equation*}
  
\item Equivalently \cite[Theorem 10.3.2]{agpr_pi}, $A$ is rank-$n^2$ Azumaya (unqualified, i.e. over its center) precisely when it satisfies the polynomial identities of $n\times n$ matrices but has no quotient satisfying the identities of $(n-1)\times (n-1)$ matrices.
\end{itemize}
We refer the reader to the sources just cited, with more specific citations accompanying the text where appropriate.

\begin{remark}\label{re:az.scl.ext}
  We will repeatedly (and henceforth tacitly) take it for granted that being Azumaya over a commutative ring is stable under scalar extension: if $A$ is Azumaya over $R$ then $A_S:=A\otimes_RS$ is Azumaya over $S$ for any commutative-ring morphism $R\to S$ \cite[Proposition 5.4.28]{agpr_pi}. 
\end{remark}

\pf{th:theta3sphiso}
\begin{th:theta3sphiso}
  That \Cref{item:samepara} implies the other conditions is clear enough: for tori transitioning between $\pm\theta$ is effected by interchanging the two unitary generators, and the realization \Cref{eq:sph2tor} settles the matter for spheres. We thus focus on the converse implications, seeking to deduce \Cref{item:samepara} from either of the other conditions.
  
  \begin{enumerate}[(I),wide]
  \item \textbf{ Rational vs. irrational parameters.} Note first that given an isomorphism \Cref{eq:ctnctn}, the deformation parameters $\theta$ and $\theta'$ are simultaneously (ir)rational.

    Indeed, if one of them ($\theta$, say) is irrational, then $C_{\theta,n}$ has infinite-dimensional {\it simple} quotients isomorphic to $A_{\theta,n}$ (simple, being a minimal tensor product of simple $C^*$-algebras: \cite[Corollary 12.12]{gvf_ncg} and \cite[Corollary IV.4.21]{tak1}). Furthermore, \Cref{eq:sph2tor} makes it clear that {\it all} infinite-dimensional simple quotients are of this form.
    
    On the other hand, if $\theta'\in \bQ$ then the simple quotients of $C_{\theta',n'}$ are plainly finite-dimensional, again by \Cref{eq:sph2tor}.

    The same argument also (more easily) handles \Cref{eq:atnatn}.
    
  \item \textbf{ The irrational case.} This is the simpler branch of the argument: we have already observed that if $\theta$ and $\theta'$ are both irrational then \Cref{eq:ctnctn} implies \Cref{eq:atnatn}, $A_{\theta,n}$ and $A_{\theta',n'}$ being the only infinite-dimensional simple quotients of $C_{\theta,n}$ and $C_{\theta',n'}$ respectively. We can then simply appeal to \cite[Theorem 3.12]{rief_canc}.

  \item \textbf{ The rational case.} I.e. we are now assuming that
    \begin{equation*}
      \theta=\frac pq\quad\text{and}\quad \theta'=\frac{p'}{q'}
    \end{equation*}
    are rational (in lowest terms, as depicted).

    We focus mostly on the implication \Cref{item:3sphiso} $\Longrightarrow$ \Cref{item:samepara}, as it will turn out to be more elaborate, indicating along the way where and how the argument changes so as to also deliver \Cref{item:2toriso} $\Longrightarrow$ \Cref{item:samepara}. There is no harm in assuming $n=n'$ throughout, as \Cref{th:recovermatsize} allows. 
    

    An isomorphism \Cref{eq:ctnctn} induces isomorphisms between
    \begin{itemize}[wide]
    \item the centers $C(\bS^3)$ of the two $C^*$-algebras (\cite[Proposition 4.5(1)]{cp_cont-equiv_xv3});

    \item and hence also between the $q\times q$ matrix bundles over the Azumaya loci
      \begin{equation*}
        \bS^3\setminus\left(\sqcup\text{ two circles}\right)
        \cong
        \bT^2\times(0,1)
      \end{equation*}
      of $C_{\theta,n}$ and $C_{\theta',n'}$ (\cite[Proposition 4.5(3)]{cp_cont-equiv_xv3}).
    \end{itemize}
    
    Over those Azumaya loci, $C_{\theta,n}$ and $C_{\theta',n'}$ are isomorphic $qn\times qn$ and $q'n'\times q'n'$ matrix bundles over $X=\bT^2\times (0,1)$. We thus have $qn=q'n'$, and because
    \begin{itemize}
    \item matrix bundles are classifiable homotopically \cite[\S 18.2.4]{hjjm_bdle};
    \item $X$ deformation-retracts onto its slice
      \begin{equation*}
        \bT^2\cong \bT^2\times \left\{\frac 12\right\}\subset X;
      \end{equation*}
    \item and over that slice the two algebras are $A_{\theta,n}$ and $A_{\theta',n'}$ respectively,      
    \end{itemize}
    the implication \Cref{item:2toriso} $\Longrightarrow$ \Cref{item:samepara} reduces to \Cref{item:3sphiso} $\Longrightarrow$ \Cref{item:samepara}, on which we henceforth focus. We have been assuming that $n=n'$, so that $q=q'$ because $qn=q'n'$. The rest is simple conceptually, but requires a bit of unwinding of the attendant bundle-classification theory.    
    
    $qn\times qn$-matrix bundles on $X=\bT^2\times (0,1)$ are classified by homotopy classes of maps $X\to BPU(qn)$ into the {\it classifying space} \cite[\S 7.2, Definition 2.7]{hjjm_bdle} of the projective $qn\times qn$ unitary group (\cite[\S 18.3, Assertion 3.2 and Remark 3.3]{hjjm_bdle}). Writing $[-,-]$ for homotopy classes of maps \cite[\S 6.4, Definition 4.2]{hjjm_bdle}, we have
    \begin{equation*}
      [X,-]\cong [\bT^2,-],
    \end{equation*}
    so we may as well assume we are working with the 2-torus.

    We have an exact sequence
    \begin{equation}\label{eq:bsu2zqn}
      [\bT^2,\ BSU(qn)]
      \xrightarrow{\quad}
      [\bT^2,\ BPU(qn)]
      \xrightarrow{\quad}
      [\bT^2,\ B^2\bZ/qn]
      \cong
      H^2(\bT^2,\ \bZ/qn)
      \cong
      \bZ/qn
    \end{equation}
    \cite[Corollary 3.2.7]{koch_bord} attached to the top row of \cite[\S 18.3.6]{hjjm_bdle}, where $B^2$ denotes the iterated classifying-space construction discussed in \cite[\S 7.4.6]{hjjm_bdle} for {\it abelian} groups (such as $\bZ/qn$). The second arrow in \Cref{eq:bsu2zqn} is an {\it embedding}, given that the leftmost term $[\bT^2,\ BSU(qn)]$ is trivial: it classifies degree-0 rank-$qn$ vector bundles on the 2-torus, and those are trivial \cite[p.2, Proposition]{thad_introtop}. In other words, $qn\times qn$-matrix-algebra bundles over the 2-torus are classified by the characteristic class denoted by $\beta_{qn}$ in \cite[\S 18.3.7]{hjjm_bdle} (our $qn$ being that source's $n$).

    Denote by $\widetilde{(-)}_q$ inverses modulo $q$ and similarly for $q'$. In our case, focusing on the un-primed side, the $qn\times qn$-matrix bundle in question is of the form $\cE\otimes \cE^*$ for a vector bundle $\cE\cong \cF\otimes \bC^n$ with rank-$q$ degree-$\widetilde{p}_q$ $\cF$, and the class $\beta_{qn}$ is the image of the Chern class
    \begin{equation*}
      n\widetilde{p}_q
      \in
      [\bT^2,\ B^2\bZ]
      \cong
      H^2(\bT^2,\ \bZ)
      \cong
      \bZ
    \end{equation*}
    through the natural map $\bZ\to \bZ/nq$ (arising from the vertical maps in \cite[\S 18.3.4]{hjjm_bdle}) so all in all, the isomorphism of matrix bundles simply amounts to
    \begin{equation*}
      \widetilde{p}_q n = \pm \widetilde{p'}_q n\mod qn
    \end{equation*}
    (the $\pm$ depending on whether the homeomorphism between the two spheres preserves or reverses orientations). This is in turn equivalent to
    \begin{equation*}
      qn\ |\ \widetilde{p}_q n\pm \widetilde{p'}_q n
      \iff
      q\ |\ \widetilde{p}_q\pm \widetilde{p'}_q
      \iff
      q\ |\ p\pm p'
      \iff
      \theta+\bZ = \pm \theta'+\bZ.
    \end{equation*}
  \end{enumerate}
  This concludes the proof. 
\end{th:theta3sphiso}

\begin{remarks}\label{res:charclasses}
  \begin{enumerate}[(1),wide]
  \item Note in particular that the obvious analogue of \cite[Theorem 3]{rief_irrat} does in fact go through for rational deformation parameters, giving an extension of \cite[Theorem 1.1]{hks_erg}.

  \item The literature on the classification of non-commutative tori up to isomorphism or Morita equivalence is extensive: the reader can consult, for instance, \cite{rief_case,bursz_surv,el_morita-def} and their references.

  \item\label{item:bsu2bzbz2} The proof of \Cref{th:theta3sphiso} takes for granted the following discussion on Chern classes and their relation to matrix-algebra-bundle characteristic classes.

    The fibration sequence associated (as in \cite[\S 18.3.6]{hjjm_bdle}) to the exact sequence
    \begin{equation*}
      1\xrightarrow{}SU(n)\xrightarrow{\quad}U(n)\xrightarrow{\ \det\ }\bS^1\xrightarrow{} 1
    \end{equation*}
    and (part of) the top portion of the diagram in \cite[\S 18.3.6]{hjjm_bdle} fit together into a commutative diagram
    \begin{equation}\label{eq:bsu2bzbz2}
      \begin{tikzpicture}[auto,baseline=(current  bounding  box.center)]
        \path[anchor=base] 
        (0,0) node (l) {$BSU(n)$}
        +(3,.5) node (u) {$BU(n)$}
        +(3,-.5) node (d) {$BPU(n)$}
        +(6,1) node (um) {$B\bS^1$}
        +(6,-1) node (dm) {$B^2\bZ/n$}
        +(9,.5) node (ur) {$K(\bZ,2)$}
        +(9,-.5) node (dr) {$K(\bZ/n,2)$,}
        ;
        \draw[->] (l) to[bend left=6] node[pos=.5,auto] {$\scriptstyle $} (u);
        \draw[->] (l) to[bend right=6] node[pos=.5,auto,swap] {$\scriptstyle $} (d);
        \draw[->] (u) to[bend left=6] node[pos=.5,auto] {$\scriptstyle B\det$} (um);
        \draw[->] (um) to[bend left=6] node[pos=.5,auto] {$\scriptstyle \cong$} (ur);
        \draw[->] (d) to[bend right=6] node[pos=.5,auto,swap] {$\scriptstyle $} (dm);
        \draw[->] (dm) to[bend right=6] node[pos=.5,auto,swap] {$\scriptstyle \cong$} (dr);
        \draw[->] (u) to[bend left=0] node[pos=.5,auto] {$\scriptstyle $} (d);
        \draw[->] (um) to[bend left=0] node[pos=.5,auto] {$\scriptstyle $} (dm);
        \draw[->] (ur) to[bend left=0] node[pos=.5,auto] {$\scriptstyle $} (dr);
      \end{tikzpicture}
    \end{equation}
    where
    \begin{itemize}[wide]

    \item $K(G,m)$ denotes, as usual \cite[Definition 9.6.1]{hjjm_bdle}, the $m^{th}$ {\it Eilenberg-MacLane space} with homotopy group $\pi_n\cong G$ and vanishing homotopy in other degrees;
      
    \item the left-hand vertical map is the obvious one, resulting from the quotient $U(n)\to PU(n)$ that mods out the center;

    \item the middle vertical arrow is obtained by applying the classifying-space (homotopy-)functor (\cite[\S 3.7]{rs_dgmp-2}, \cite[Proposition 8.1]{mitch_notes-2011}) $B(-)$ to the map
      \begin{equation*}
        \bS^1\xrightarrow{\quad} B\bZ/n
      \end{equation*}
      classifying the $\bZ/n$-bundle on the circle obtained via the $n$-fold cover
      \begin{equation*}
        \bS^1
        \lhook\joinrel\xrightarrow{\quad\text{central embedding}\quad}
        U(n)
        \xrightarrowdbl{\quad\det\quad}
        \bS^1;
      \end{equation*}
      
    \item and the right-hand vertical map induced by $\bZ\to \bZ/n$, well-defined only upon choosing a generator for $\bZ/n$ (so a choice is involved). 
    \end{itemize}
    For a space $X$ the upper composition
    \begin{equation*}
      \left(\text{$n$-bundles on $X$}\right)
      \xrightarrow[\cong]{\text{\cite[Assertion 18.3.2]{hjjm_bdle}}}      
      [X,\ BU(n)]
      \xrightarrow{\quad}
      [X,\ K(\bZ,2)]
      \cong
      H^2(X,\bZ)
    \end{equation*}
    simply associates the first Chern class $c_1(\cE)\in H^2(X,\bZ)$ to a vector bundle $\cE$ on $X$, whereas the bottom composition
    \begin{equation*}
      \left(\text{$n^2$-matrix bundles on $X$}\right)
      \xrightarrow[\cong]{\text{\cite[Assertion 18.3.2]{hjjm_bdle}}}      
      [X,\ BPU(n)]
      \xrightarrow{\quad}
      [X,\ K(\bZ/n,2)]
      \cong
      H^2(X,\bZ/n)
    \end{equation*}
    similarly associates the class $\beta_n(\cA)$ of \cite[\S 18.3.7]{hjjm_bdle} to a matrix-algebra bundle $\cA$.

    In conclusion, the Chern class $c_1(\cE)\in H^2(X,\bZ)$ of a vector bundle gets mapped to the $H^2(X,\bZ/n)$-characteristic class of the corresponding matrix-algebra bundle $\cE\otimes \cE^*$.

  \item\label{item:char2class} We will refer to the class $\beta_n\in H^2(X,\bZ/n)$ associated in \cite[\S 18.3.7]{hjjm_bdle} to an $n\times n$ matrix-algebra bundle on $X$ as {\it the characteristic 2-class} of the matrix bundle, to distinguish it from the {\it Dixmier-Douady} (3-)class $\alpha\in H^3(X,\bZ)$ also discussed in \cite[\S 18.3.7]{hjjm_bdle}. 
  \end{enumerate}  
\end{remarks}

\section{Azumaya theory for higher quantum spheres}\label{se:high.dim}

Much can be said about higher-dimensional rational tori $C(\bT^n_{\theta})=A_{\theta}^n$ \cite[\S 12.2]{gvf_ncg}  as well. These algebras are also Azumaya, (essentially) by \cite[Proposition 7.2]{dcp_qg}. The setup there is more algebraically-oriented so the discussion requires some translation, but nothing particularly problematic. The deformation-parameter data in \cite[\S 7]{dcp_qg} consists of
\begin{itemize}
\item a primitive $\ell^{th}$ root of unity $q$ in the ground field (which for us is $\bC$);
\item and a skew-symmetric $n\times n$ matrix $H=(h_{ij})\in M_n(\bZ)$, which will turn out to be an integer multiple of our $\theta$. 
\end{itemize}

The $n$ invertible generators $x_i$ of loc.cit. are required to skew-commute in the sense that
\begin{equation}\label{eq:xgens}
  x_j x_k = q^{h_{jk}}x_k x_j;
\end{equation}
compare this with the usual skew-commutation relation \cite[p.193]{rief_case}
\begin{equation}\label{eq:ugens}
  u_k u_j = e^{2\pi i \theta_{jk}}u_j u_k
\end{equation}
for the unitary generators of $A_{\theta}^n$. Now, if $\ell$ is a positive integer divisible by the denominators of all $\theta_{jk}$, then
\begin{equation*}
  x_{\bullet} = u_{\bullet}
  ,\quad
  q = e^{-\frac{2\pi i}{\ell}}
  ,\quad
  H = \ell\theta\in M_n(\bZ)
\end{equation*}
will effect the transition between \Cref{eq:xgens,eq:ugens}.

\cite[Proposition 7.2]{dcp_qg} (or rather the appropriately-translated version) then shows that $A_{\theta}^n$ is Azumaya of rank
\begin{equation}\label{eq:whatish}
  h=h_{\theta}:=\text{cardinality of}\quad\mathrm{range}\left(\bZ^n\xrightarrow{\quad H=\ell\theta\quad}(\bZ/\ell)^n\right),
\end{equation}
where a matrix is regarded as an operator in the usual way.

Note that although there was a choice in selecting $\ell$, $h$ itself will not depend on that choice: scaling $\ell$ to, say, $d\ell$ will also scale $H$ by $d$, thus preserving the size of its image. In fact, the number \Cref{eq:whatish} has the following alternative, direct description in terms of the matrix $\theta$ alone:

\begin{lemma}\label{le:whatish}
  For a rational skew-symmetric $\theta\in M_n(\bQ)$ the number \Cref{eq:whatish} can be recovered as the index
  \begin{equation}\label{eq:htheta}
    h_{\theta} = [(\bZ^n+\mathrm{im}~\theta)\ :\ \bZ^n].
  \end{equation}
\end{lemma}
\begin{proof}
  We can assume \cite[\S IX.5.1, Th\'eor\`eme 1]{bourb_alg-09} that $\theta$ is block-diagonal, consisting of a zero block and one of the form
  \begin{equation}\label{eq:ddmat}
    \begin{pmatrix}
      \phantom{-}0&D\\
      -D&0
    \end{pmatrix}    
  \end{equation}
  for a diagonal matrix
  \begin{equation*}
    D=\mathrm{diag}\left(\frac{p_1}{q_1},\ \frac{p_2}{q_2},\ \cdots\right)
  \end{equation*}
  with lowest-terms non-zero entries. \Cref{eq:whatish,eq:htheta} are now easily seen to both be equal to $\prod_i q_i^2$. 
\end{proof}

\begin{remark}
  The proof of \Cref{le:whatish} also makes it plain that \Cref{eq:whatish} is indeed a square, as implicit in the discussion: it is the common dimension of the matrix-algebra fibers of an Azumaya algebra. 
\end{remark}

Still assuming $\theta$ rational, we have 

\begin{notation}\label{not:tuple.pows}
  Consider a skew-symmetric matrix $\theta\in M_n(\bR)$ (usually rational).

  \begin{enumerate}[(1),wide]
  \item For the tuple $\mathbf{u}:=(u_i)_{i=1}^n$ of generators of $A^n_{\theta}$ and integers $\mathbf{m}=(m_i)_{i=1}^n$ we write
    \begin{equation*}
      \mathbf{u}^{\mathbf{m}}
      :=
      \prod_{i=1}^n u_i^{m_i}.
    \end{equation*}
    The product is to be understood as ordering the indices increasingly unless otherwise specified, but at no point will the ordering in fact matter: we will mostly be interested in $C^*$-subalgebras of $A^n_{\theta}$ generated by such products, and a reordering simply scales by a modulus-1 complex number.

  \item The \emph{integral kernel} of $\theta$ is 
    \begin{equation*}
      \theta^{\perp}
      :=
      \left\{\mathbf{m}\in \bZ^n\ :\ \theta\mathbf{m}\in \bZ^n\right\},
    \end{equation*}
    regarding $\theta$ as a linear endomorphism on $\bR^n$.

  \item For $F\subseteq [n]$ write
    \begin{itemize}[wide]
    \item $\theta|_F$ for the $|F|\times |F|$ matrix consisting the $\theta$-entries in $F$-labeled rows and columns;

      
    \item and $h_{\theta,F}:=h_{\theta|_F}$, with $h_{\bullet}$ as in \Cref{eq:htheta}.
    \end{itemize}
  \end{enumerate}
\end{notation}

The realization of $A^n_{\theta}$ as a cocycle twist of the group algebra of $\bZ^n$ \cite[\S 2.2]{el_morita-def}, coupled with the proof of \cite[Lemma 2.3]{el_kproj}, describing the center of such a twisted group algebra, gives
\begin{equation}\label{eq:rat.tor.cent}
  Z(A^n_{\theta})
  =
  C^*\left(
    \mathbf{u}^{\theta^{\perp}}
  \right)
  :=
  C^*\left(
    \mathbf{u}^{\mathbf{m}}
    \ :\
    \mathbf{m}\in \theta^{\perp}
  \right)
  \overset{\theta\in M_n(\bQ)}{\cong}  
  C^*(\bZ^n)
  \cong
  C(\bT^n)
\end{equation}
By the Azumaya claim we can once more realize $A^n_{\theta}$ as sections of a bundle of dimension-$h_{\theta}$ matrix algebras over $\bT^n$ for the $h_{\theta}$ of \Cref{eq:htheta}.


\begin{remark}\label{re:highnotflat}
  In dealing with vector bundles over higher-dimensional tori, one additional complication not visible for $n=2$ (or $n=3$, for that matter) is the fact that the cohomological classification recalled in \cite[Remark 4.4(1)]{cp_cont-equiv_xv3} no longer goes through: as observed in \cite[introductory remarks, p.247]{oda_vb-ab} and elaborated in \cite[\S 3]{oda_vb-ab}, there are non-flat rank-2 vector bundles on complex 2-dimensional tori (i.e. {\it four}-dimensional when regarded as real tori) with vanishing Chern classes.

  Such pathologies thus occur as soon as one can reasonably expect them: whenever the torus and bundle are large enough for both first Chern classes $c_1$ and $c_2$ to come into play, i.e. the bundle has rank $\ge 2$ and the torus dimension $\ge 4$. 
\end{remark}

Incidentally, the proof of \Cref{le:whatish} also shows that rational non-commutative tori are decomposable into small-dimensional pieces:

\begin{lemma}\label{le:hightoridec}
  For a skew-symmetric matrix $\theta\in M_n(\bQ)$ the corresponding non-commutative torus algebra $A^n_{\theta}$ decomposes as
  \begin{equation*}
    A^n_{\theta}\cong C(\bT^k)\otimes\bigotimes_{i=1}^t A^{2}_{\theta_i}
  \end{equation*}
  for rationals $\theta_i\in \bQ$ and $n=k+2t$.  
\end{lemma}
\begin{proof}
  Immediate from the observation made in passing in the proof of \Cref{le:whatish} that after a change of (integer) basis $\theta$ is a direct sum of a ($k\times k$, say) zero matrix and a ($2t\times 2t$) matrix of the form \Cref{eq:ddmat}.
\end{proof}

It furthermore follows that \Cref{re:highnotflat} (and \cite[Remark 4.4(1)]{cp_cont-equiv_xv3}) notwithstanding, the bundles relevant to the Azumaya structure of $A^n_{\theta}$ are relatively well-behaved. To make sense of the statement, recall \cite[\S\S I.2 and I.4]{kob_cplx} the notion of a {\it projectively flat} vector bundle over a space $X$ (either plain, i.e. with structure group $G=GL(r)$, or {\it Hermitian} \cite[\S I.4]{kob_cplx}, i.e. with structure group $G=U(r)$):
\begin{itemize}[wide]
\item consider the {\it principal} $G$-bundle $P\to X$ associated to the vector bundle $E$ in question, via the usual \cite[\S 18.2, Assertion 2.3]{hjjm_bdle} correspondence;

\item the quotient $\overline{P}:=P/Z(G)$ by the center of $G$ (scalar matrices, so either $\bC^{\times}$ when $G=GL(r)$ or $\bS^1$ for $G=U(r)$) is then a principal bundle over the {\it projective} group
  \begin{equation*}
    PG:=G/Z(G)=PGL(r)\text{ or }PU(r);
  \end{equation*}

\item and the original bundle is projectively flat if $\overline{P}$ is flat in the usual sense \cite[Proposition I.2.6]{kob_cplx}. 
\end{itemize}

\begin{proposition}\label{pr:isprojfl}
  For skew-symmetric $\theta\in M_n(\bQ)$ we have
  \begin{equation*}
    A^n_{\theta}
    \cong
    \End(\cE)
    =
    \Gamma(\cE\otimes \cE^*)
    :=
    \left(\text{continuous sections of the bundle }\cE\otimes \cE^*\right)
  \end{equation*}
  for a projectively flat bundle $\cE$ on $\bT^n\cong\Max(Z(A^n_{\theta}))$ of rank $\sqrt{h_{\theta}}$ for $h_{\theta}=\text{\Cref{eq:htheta}}$.

  In particular, $\cE\otimes \cE^*$ itself is a flat matrix-algebra bundle. 
\end{proposition}
\begin{proof}
  That the rank is as claimed has already been noted above, as a consequence of \cite[Proposition 7.2]{dcp_qg}. \Cref{le:hightoridec} reduces the problem to non-commutative 2-tori. Indeed, it shows that $\cE$ decomposes as an {\it external (or exterior) tensor product} \cite[\S 2.6]{at_k}, and external tensor products preserve projective flatness: the latter is definable \cite[Proposition I.2.8]{kob_cplx} as the existence of a connection whose curvature takes scalar values, and there is a simple formula for the curvature of the tensor product of the natural two connections \cite[\S I.5, (5.15)]{kob_cplx}.
  
  The conclusion now follows from the fact that {\it all} complex vector bundles on a 2-torus (and indeed a compact orientable surface) are projectively flat: being completely classified by rank and Chern class (\cite[Proposition, p.2]{thad_introtop} again) a rank-$q$ bundle splits as the sum between a line bundle and a trivial rank-$(q-1)$ bundle, etc.
  
  As for the last claim on flatness, it follows from the general remark that $\cE\otimes \cE^*$ is flat whenever $\cE$ is projectively flat \cite[Propositions I.2.9 and/or I.4.23]{kob_cplx}.
\end{proof}

\begin{remark}\label{re:isflatref}
  The flatness of the relevant $q\times q$ matrix bundle for $n=2$ and $\theta=\frac pq$ is plain from its direct construction in the proof of \cite[Proposition 12.2]{gvf_ncg} (or \cite[\S 2]{hks_erg}, on which the later account is based): the bundle is obtained as the (total space of the) quotient
  \begin{equation*}
    (\bT^2\times M_q)/(\bZ/q)^2
  \end{equation*}
  through the diagonal action, where
  \begin{equation*}
    (\bZ/q)^2\subset \bT^2
  \end{equation*}
  acts on the torus via translation by the $q$-torsion subgroup and on $M_q$ as described in loc.cit. (the two generators act by conjugation by two order-$q$ unitary matrices that commute up to scalars, etc.). 
\end{remark}

A few reminders will help make sense of the statement of \Cref{th:s2n1thetapi}, a higher-dimensional generalization of \cite[Proposition 4.5]{cp_cont-equiv_xv3}.

\begin{itemize}[wide]
\item The \emph{PI-degree} (\cite[Definitions A.7.1.8 and B.4.15]{df_pi}, \cite[\S 1]{zbMATH04101395}) of a PI algebra is the largest $n$ for which the polynomial identities of $A$ are among those of $M_n$. Rank-$n^2$ Azumaya algebras, for instance, have PI-degree $n$ (as observed \cite[post Definition B.4.15]{df_pi}). 

  The slight definition variations one typically encounters in the literature will not make a difference here, so the above will do.
  
\item The {\it Azumaya locus} of an algebra $A$ consists of those maximal ideals
  \begin{equation*}
    \fm\in \Max \left(Z:=Z(A)\right)
    :=
    \text{maximal spectrum of the center }Z
  \end{equation*}
  for which the \emph{localization} \cite[Example (1) post Corollary 3.2]{am_comm} $A_{\fm}$ is Azumaya over $Z_{\fm}$.
  
  There are again minor departures from this precise setup in the literature (one might, for instance, consider all \emph{prime} ideals with the requisite property \cite[\S 1]{zbMATH05905742}, impose additional finiteness/primality constraints on the algebra $A$ \cite[\S III.1.7]{bg_algqg}, etc.). 
\end{itemize}

Set
\begin{equation}\label{eq:iy.ideal}
  I_Y:=\left\{f|_Y\equiv0\right\}
  \quad
  \trianglelefteq
  \quad
  \mathrm{Cont}_{\partial}
  \left(
    \bS_+^{n-1}\xrightarrow{\quad}C\left(\bT^n_{\theta}\right)
  \right)
  \quad
  \overset{\text{\Cref{eq:sph2tor}}}{\cong}
  \quad
  C_{\theta}
  ,\quad
  Y=\overline{Y}\subseteq \bS^{n-1}_+
\end{equation}
and
\begin{equation}\label{eq:ct.zt.res}
  C_{\theta}|_Y:=C_{\theta}/I_Y
  ,\quad
  Z_{\theta}|_Y:=Z_{\theta}/\left(I_Y\cap Z_{\theta}\right).
\end{equation}

For subsets $S\subseteq \bZ^n$ and $F\subseteq [n]$ we write
\begin{equation*}
  S_{\downarrow F}
  :=
  \left\{
    \mathbf{s}=(s_i)_{i=1}^n\in S
    \ :\
    s_i=0,\ \forall i\not\in F
  \right\};
\end{equation*}
the portion of $S$ \emph{supported on $F$}, in other words. This applies to the symbol $\theta^{\perp}_{\downarrow F}$ employed below. 

\begin{theorem}\label{th:s2n1thetapi}
  Let $\theta\in M_n(\bQ)$ be an $n\times n$ skew-symmetric rational matrix for some fixed $n\ge 2$.
  \begin{enumerate}[(1),wide]
    
  \item\label{item:th:s2n1thetapi:s2n1-fincent} The center of the non-commutative sphere algebra $C_{\theta}:=C(\bS^{2n-1}_{\theta})$ is 
    \begin{align*}
      Z_{\theta}
      &:=Z\left(C\left(\bS_{\theta}^{2n-1}\right)\right)\\
      &\cong
        \mathrm{Cont}_{\partial}\left(\bS^{n-1}_{+}\xrightarrow{\quad} Z(C(\bT^n_{\theta}))\right)\numberthis\label{eq:cents2n1theta}\\
      &=
        \left\{
        \bS^{n-1}_+
        \xrightarrow[\text{cont.}]{\quad f\quad}
        C^*\left(\mathbf{u}^{\theta^{\perp}}\right)
        \ :\
        f\left(\bS^{n-1}_{+,F}\right)
        \in
        C^*\left(\mathbf{u}^{\theta^{\perp}_{\downarrow F}}\right),\ \forall F\subseteq [n]
        \right\}.
    \end{align*}   
    
  \item\label{item:th:s2n1thetapi:brnch.cvr.sph} The spectrum $X_{\theta}:=\Max(Z_{\theta})$ branch-covers the classical sphere
    \begin{equation}\label{eq:lcm.sph}
      \bS^{2n-1}
      \cong
      \Max
      \left\{
        \bS^{n-1}_+
        \xrightarrow[\text{cont.}]{\quad f\quad}
        C^*\left(u_i^{q_i}\right)
        \ :\
        f\left(\bS^{n-1}_{+,F}\right)
        \in
        C^*\left(u_i^{q_i},\ i\in F\right),\ \forall F\subseteq [n]
      \right\}
    \end{equation}
    for
    \begin{equation*}
      q_i:=\lcm\left(\text{lowest-term denominators of }\theta_{ij},\ j\in [n]\right). 
    \end{equation*}
    
  \item\label{item:th:s2n1thetapi:s2n1-ispi} $C_{\theta}$ embeds into $M_{\sqrt{h_{\theta}}}\left(Z_{\theta}\right)$ with $h_{\theta}$ as in \Cref{eq:htheta} and is not a $Z_{\theta}$-subalgebra of $M_n(Z_{\theta})$ for any smaller $n$. In particular, $C_{\theta}$ is a PI algebra of PI-degree $\sqrt{h_{\theta}}$.
    
    
    
  \end{enumerate}  
\end{theorem}
\begin{proof}  
  \begin{enumerate}[label={},wide]
  \item\textbf{\Cref{item:th:s2n1thetapi:s2n1-fincent}} We have already recalled \Cref{eq:rat.tor.cent} that the centers $Z(C(\bT^n_{\theta}))$ of the non-commutative torus algebras are as claimed, so the conclusion follows from the identification \Cref{eq:sph2tor}.

  \item\textbf{\Cref{item:th:s2n1thetapi:brnch.cvr.sph}} That the right-hand side of \Cref{eq:lcm.sph} is indeed a sphere is a simple topology exercise (the classical counterpart to \Cref{eq:sph2tor}), given that the $u_i^{q_i}\in A^n_{\theta}$ are central and hence the generators of a $C(\bT^n)$. The claim is that the map
    \begin{equation*}
      X_{\theta}
      \xrightarrowdbl[\quad\text{onto}\quad]{\quad\pi\quad}
      \bS^{2n-1}
    \end{equation*}
    dualizing the embedding
    \begin{equation}\label{eq:s2x}
      \begin{aligned}
        C\left(\bS^{2n-1}\right)
        &\cong
          \mathrm{Cont}_{\partial}
          \left(
          \bS^{n-1}_{+}
          \xrightarrow{\quad}
          Z_{\downarrow}\left(A^n_{\theta}\right)
          :=
          C^*\left(u_i^{q_i}\right)
          \right)\\
        &:=
          \left\{
          \bS^{n-1}_+
          \xrightarrow[\text{cont.}]{\quad f\quad}
          Z_{\downarrow}\left(A^n_{\theta}\right)
          \ :\
          f\left(\bS^{n-1}_{+,F}\right)
          \in
          C^*\left(u_i^{q_i},\ i\in F\right),\ \forall F\subseteq [n]
          \right\}\\
        &\lhook\joinrel\xrightarrow{\qquad}
          \text{\Cref{eq:cents2n1theta}}
          =:C(X_{\theta})
      \end{aligned}      
    \end{equation}
    is a \emph{branched cover} in the sense of \cite[\S 1]{pt_brnch}: an open surjection of compact Hausdorff spaces, with a finite upper bound on the cardinalities of the fibers $\pi^{-1}(p)$, $p\in \bS^{2n-1}$. Surjectivity is not at issue, so it is the two other requirements that require attention. Rather than attack that classical statement directly, we appeal to the theory of \emph{non-commutative branched covers} developed in \cite{pt_brnch,bg_cx-exp}, revolving around the notion of a $C^*$ \emph{conditional expectation} \cite[Definition II.6.10.1]{blk}.

    We have embeddings
    \begin{equation*}
      Z_{\downarrow}\left(A^n_{\theta}\right)
      \lhook\joinrel\xrightarrow{\quad}
      Z\left(A^n_{\theta}\right)
      \lhook\joinrel\xrightarrow{\quad}
      C^*\left(\bZ^n\right)
      \cong
      C^*\left(t_i,\ i\in [n]\right)
      ,\quad
      t_i^{q_i}=u_i^{q_i}.
    \end{equation*}
    The usual \cite[Proposition 8.5]{pis_os} (plainly of \emph{finite index} \cite[Definition 2]{fk_fin-ind}) expectation
    \begin{equation*}
      C^*\left(t_i,\ i\in [n]\right)
      \xrightarrowdbl{\quad E\quad}
      Z_{\downarrow}\left(A^n_{\theta}\right)
    \end{equation*}
    restricts to another such (also $E$) on the intermediate $Z\left(A^n_{\theta}\right)$, compatible with the inclusions of the subgroups generated by $t_i$, $i\in F$ for $F\subseteq [n]$. 

    We then have a finite-index 
    \begin{equation*}
      \mathrm{Cont}
      \left(
        \bS^{n-1}_{+}
        \xrightarrow{\quad}
        Z\left(A^n_{\theta}\right)
      \right)
      \xrightarrowdbl{\quad E\circ\quad}
      \mathrm{Cont}
      \left(
        \bS^{n-1}_{+}
        \xrightarrow{\quad}
        Z_{\downarrow}\left(A^n_{\theta}\right)
      \right),
    \end{equation*}
    hence also (because of the noted $F$-compatibility) an analogous expectation between $\mathrm{Cont}_{\partial}$ function algebras. These, however, are exactly the two extremes of \Cref{eq:s2x}, whence the conclusion by \cite[Theorem 1.4]{bg_cx-exp}.

  \item\textbf{\Cref{item:th:s2n1thetapi:s2n1-ispi}} \Cref{eq:sph2tor} and \Cref{pr:isprojfl} ensure that the irreducible $C_{\theta}$-representations are all at most $\sqrt{h_{\theta}}$-dimensional. That estimate, moreover, is the best possible (in the language of \cite[Definition IV.1.4.1]{blk}, $C_{\theta}$ is sharply \emph{$\sqrt{h_{\theta}}$-subhomogeneous}). To verify this last sharpness assertion, quotient out the ideal $I_p:=I_{\{p\}}$ of \Cref{eq:iy.ideal} in \Cref{eq:sph2tor} to obtain a surjection
    \begin{equation*}
      C_\theta
      \xrightarrowdbl{\quad}
      A^n_{\theta}
      \xrightarrowdbl{\quad\text{\Cref{pr:isprojfl}}\quad}
      M_{\sqrt{h_{\theta}}}.
    \end{equation*}
    $C_{\theta}$ thus embeds \cite[\S 2.7.3]{dixc} into a product of matrix algebras $M_n$, $n\le \sqrt{h_{\theta}}$ with at least one $M_{\sqrt{h_{\theta}}}$ quotient, and the conclusion follows.
  \end{enumerate}
\end{proof}

In particular:

\begin{corollary}\label{cor:azumaya.rare}
  A non-commutative sphere algebra $C_{\theta}$ is Azumaya if and only if it is commutative, i.e. $\theta\in M_n(\bZ)$.
\end{corollary}
\begin{proof}
  This is immediate from \Cref{th:s2n1thetapi} (and its proof): $\theta\not\in M_n(\bZ)$ is equivalent to $h_{\theta}>1$, and it remains to observe that if that condition holds then $C_{\theta}$ cannot be Azumaya, for on the one hand it has an $h_{\theta}$-dimensional matrix algebra quotient, while on the other hand $C_{\theta}|_p$ is abelian for any of the $n$ vertices $p\in \bS_+^{n-1}$. 
\end{proof}

The branch-covering qualification in \Cref{th:s2n1thetapi}\Cref{item:th:s2n1thetapi:brnch.cvr.sph} is not redundant: as soon as $n\ge 3$ the center $Z_{\theta}$ of \Cref{eq:cents2n1theta} need not be (the function algebra of) a sphere. 

\begin{example}\label{ex:s5.non.sph}
  Set
  \begin{equation*}
    \theta
    :=
    \left(
      \begin{array}{rrr}
        0 & 1/2 & 1/2  \\
        -1/2 & 0 &  1/2 \\
        -1/2 & -1/2 & 0
      \end{array}
    \right)
    \in
    M_3(\bQ).
  \end{equation*}
  Denoting by $u_i$, $i\in [3]$ the generators of $A^3_{\theta}$, the sphere $\bS^{2n-1}=\bS^5$ of \Cref{th:s2n1thetapi}\Cref{item:th:s2n1thetapi:brnch.cvr.sph} is
  \begin{equation*}
    \bS^5
    \cong
    \Max \mathrm{Cont}_{\partial}
    \left(
      \bS_+^2
      \xrightarrow{\quad}
      C^*\left(u_i^2,\ i\in [3]\right)
    \right)
  \end{equation*}
  while
  \begin{equation*}
    X_{\theta}
    \cong
    \Max \mathrm{Cont}_{\partial}
    \left(
      \bS_+^2
      \xrightarrow{\quad}
      Z\left(A^3_{\theta}\right)
      =
      C^*\left(u_i^2,\ i\in [3],\ u_1 u_2 u_3\right)
    \right).
  \end{equation*}
  The latter space thus consists of two 5-spheres glued along the 3-dimensional complex
  \begin{equation}\label{eq:3sph.chn}
    \Max \mathrm{Cont}_{\partial}
    \left(
      \partial \bS_+^2
      \xrightarrow{\quad}
      Z\left(A^3_{\theta}\right)
    \right)
  \end{equation}
  (one copy of which is embedded in each of the two 5-spheres). Said complex can be described as follows:
  \begin{itemize}[wide]
  \item consider three 3-spheres $\bS^3_i$ indexed by $i\in \bZ/3$, regarded as total spaces of the \emph{Hopf fibration} \cite[\S 14.1.9]{td_alg-top} $\bS^3\xrightarrowdbl{\pi} \bS^2$;

  \item in each $\bS^3_i$ set $\bS^1_{i,\pm}:=\pi^{-1}\left(p_{\pm}\right)$ for antipodes $p_{\pm}\in \bS^2$;

  \item and glue $\bS^3_i$ to $\bS^3_{i+1}$ (indices modulo 3) by identifying $\bS^1_{i,+}\cong \bS^1_{i+1,-}$ to obtain \Cref{eq:3sph.chn}. 
  \end{itemize}
  $X_{\theta}$ is now easily seen not to be homeomorphic to $\bS^5$, and indeed, not a topological manifold: points in $\bS^3_0$ but off the exceptional circles $\bS^1_{0,\pm}$ have neighborhoods homeomorphic to two copies of $\bR^5$ glued along a closed $\bR^3$. The removal of that $\bR^3$ disconnects such a neighborhood, so it cannot \cite[Theorem 1.8.13]{eng_dim} be homeomorphic to a Euclidean space. 
\end{example}

\begin{remark}\label{re:cntr.not.inv}
  \Cref{ex:s5.non.sph} illustrates a qualitative distinction between quantum tori and spheres: while \emph{equivalent} $\theta,\theta'\in M_n(\bQ)$, in the sense \cite[\S IX.1.6]{bourb_alg-09} (appropriate for bilinear forms) that
  \begin{equation*}
    \exists\left(T\in \mathrm{GL}_n(\bZ)\right)
    \quad:\quad
    \theta'=T\theta T^t,
  \end{equation*}
  will produce isomorphic torus algebras $A^n_{\theta}\cong A^n_{\theta'}$, it may well be that $C_{\theta}\not\cong C_{\theta'}$. In fact, not even the centers of the latter two algebras need be isomorphic.

  Indeed, $Z_{\theta}$ will be (the function algebra of) a $(2n-1)$-sphere whenever $\theta$ is block-diagonal with blocks \Cref{eq:ddmat}, for in that case, in the notation of \Cref{th:s2n1thetapi},
  \begin{equation*}
    Z\left(A^n_{\theta}\right)
    =
    C^*\left(u_i^{q_i},\ i\in [n]\right).
  \end{equation*}
  Per \Cref{ex:s5.non.sph}, then, with that choice of $\theta$ and
  \begin{equation*}
    \theta':=
    \left(
      \begin{array}{rrr}
        0 & 1/2 & 0  \\
        -1/2 & 0 & 0 \\
        0 & 0 & 0
      \end{array}
    \right)
    =
    T\theta T^t
    ,\quad
    T:=
    \left(
      \begin{array}{rrr}
        1 & 0 & 0  \\
        0 & 1 & 0 \\
        1 & -1 & 1
      \end{array}
    \right)    
  \end{equation*}
  The respective centers $Z_{\theta,\theta'}=Z\left(C_{\theta,\theta'}\right)$ are non-isomorphic. 
\end{remark}

To follow up on \Cref{cor:azumaya.rare}, determining the Azumaya locus of $C_{\theta}$ is now also. We again need some vocabulary and notation preliminaries.


\begin{definition}\label{def:jmp.cplx}
  Let $\theta\in M_n(\bQ)$ be a rational skew-symmetric matrix and recall \Cref{not:tuple.pows}. 

  The \emph{jump (sub)complex} $\cat{jmp}_{\theta}$ of $\theta$ (just plain $\cat{jmp}$ when $\theta$ is understood) is
  \begin{equation*}
    \cat{jmp}
    =
    \cat{jmp}_{\theta}
    :=
    \left\{F\in 2^{[n]}\ :\ h_{\theta,F}<h_{\theta}\right\}.
  \end{equation*}
  Identifying \emph{simplicial complexes} and their respective \emph{geometric realizations} \cite[\S 8.1]{td_alg-top}, $\cat{jmp}_{\theta}$ is a \emph{subcomplex} (\cite[\S 8.1]{td_alg-top}, \cite[\S 3.1]{spa_at}) of the simplex $\Delta^{n-1}\cong \bS_+^{n-1}$ of \Cref{eq:bdry.cond} with vertex-set $[n]$; this justifies the term. 
\end{definition}

\begin{theorem}\label{th:ctheta.azumaya.loc}  
  The Azumaya locus of $C_{\theta}$ is the complement
  \begin{equation}\label{eq:minus.jmp}
    \left(X_{\theta}= \Max(Z_{\theta})\right)
    \setminus
    \Max\left(Z_{\theta}|_{\cat{jmp}}\right)
  \end{equation}
  in the notation of \Cref{eq:ct.zt.res}.
\end{theorem}
\begin{proof}  
  \begin{enumerate}[(I),wide]

  \item\textbf{: \Cref{eq:minus.jmp} $\supseteq$ Azumaya locus.} Suppose $p\in \Max\left(Z_{\theta}\right)$ belongs precisely to the subspaces
    \begin{equation*}
      X_{\theta,\widehat{i}}
      :=
      \Max\left(Z_{\theta}|_{\bS^{n-1}_{+,\widehat{i}}}\right)
      \subseteq
      X_{\theta}
      ,\quad
      i\in F\in 2^{[n]}.
    \end{equation*}
    The intersection
    \begin{equation*}
      p\cap C\left(\bS_+^{n-1}\right)
      \in
      \Max C\left(\bS_+^{n-1}\right)
      \cong
      \bS_+^{n-1}
    \end{equation*}
    will then belong to the interior of the face $\bS_{+,\widehat{F}}^{n-1}$, and \Cref{eq:sph2tor} specializes to
    \begin{equation*}
      C_{\theta}/\left(C_{\theta}\cdot \left(p\cap C\left(\bS_+^{n-1}\right)\right)\right)
      \cong
      C^*\left(u_i,\ i\in \widehat{F}\right)
      \cong
      A^{\left|\widehat{F}\right|}_{\theta|_F}
      \subseteq
      A^n_{\theta}. 
    \end{equation*}
    If $\widehat{F}\in \cat{jmp}$ then (by the very definition of the jump complex) this further specializes, at every maximal ideal of the center $Z\left(A^n_{\theta}\right)\cong C(\bT^n)$, to a \emph{proper} inclusion
    \begin{equation*}
      C_{\theta}/\left(C_{\theta}\cdot p\right)
      \quad      
      \cong
      \quad
      \left(M_{\sqrt{h_{\theta,\widehat{F}}}}\right)
      ^
      {\left[\left(\theta|_{\widehat{F}}\right)^{\perp}:\theta^{\perp}_{\downarrow \widehat{F}}\right]}
      \quad
      \lneq
      \quad
      M_{\sqrt{h_{\theta}}}
    \end{equation*}
    of a non-matrix algebra, with $\left(\theta|_{\widehat{F}}\right)^{\perp}\le \bZ^{\left|\widehat{F}\right|}$ regarded as a subgroup of $\bZ^n$ by padding vectors with zero entries. $C_{\theta}$, then, cannot be Azumaya at $p$, for its localization $\left(C_{\theta}\right)_p$ does not satisfy the polynomial identities of any $M_n$, $n<\sqrt{h_{\theta}}$.
    
  \item\textbf{: \Cref{eq:minus.jmp} $\subseteq$ Azumaya locus.} Every maximal ideal $p\in \text{\Cref{eq:minus.jmp}}$ is contained in the interior of some
    \begin{equation*}
      \Max\left(Z_{\theta}|_Y\right)
      \subset
      \Max\left(Z_{\theta}\right)
      ,\quad
      Y=\overline{Y}
      \subseteq
      \bS_+^{n-1}\setminus \cat{jmp}:
    \end{equation*}
    the intersection $p\cap C\left(\bS_+^{n-1}\right)$ belongs to the latter open set, and it suffices to take for $Y$ a closed neighborhood of that point still contained in that open. The conclusion now follows from the fact that for such $Y$ the algebras $C_{\theta}|_Y$ are Azumaya:
    \begin{equation*}
      C_{\theta}|Y
      \cong
      \End(\pi^*\cE)
      ,\quad
      Y\times \bT^n
      \xrightarrowdbl[\quad]{\quad\pi:=\text{$2^{nd}$ projection}\quad}
      \bT^n\cong \Max\left(Z\left(A^n_{\theta}\right)\right).
    \end{equation*}
  \end{enumerate}
  This concludes the proof of the theorem. 
\end{proof}

\section{On and around center-finiteness}\label{se:cent.fin}

Finiteness (i.e. being a finitely-generated module) over the center is well-known not to be automatic for PI algebras: \cite[Example post Proposition 5.1.3]{proc_pi}, say, is of a finitely-generated algebra satisfying the identities of $M_4$ but not embeddable in a ring finite over its center. The non-commutative sphere algebras $C_{\theta}$ themselves, as will become apparent, are another case in point. 

Subhomogeneous $C^*$-algebras $A$ (such as $C_{\theta}$) admit (unital) embeddings
\begin{equation*}
  A
  \le
  M_n^I
  \cong
  M_n\left(C(\beta I)\right)
  ,\quad
  \beta(-)
  :=
  \text{\emph{Stone-\v{C}ech compactification} \cite[\S 6.5]{gj_rings} of $I$},
\end{equation*}
as noted in the proof of \Cref{th:s2n1thetapi} (essentially by \cite[\S 2.7.3]{dixc}).
This though, will also not (generally) suffice to ensure center-finiteness, per \Cref{ex:not.cent.fin}: the latter will indeed not even be \emph{topologically} center-finite, in the sense of containing a (norm-topology-)dense module over its center. Naturally, rings sandwiched as $R\le A\le M_n(R)$ for \emph{Noetherian} commutative $R$ are center-finite, so \Cref{ex:not.cent.fin} is in a sense a manifestation of the non-Noetherianness of the continuous-function algebras $C(X)$ involved. 


\begin{example}\label{ex:not.cent.fin}
  \cite[Example 3.5]{zbMATH05172034} (also \cite[Example 3.6]{bg_cx-exp}) provides a $C^*$-algebra $A$ equipped with a central morphism $C(X)\to A$ (a $C(X)$-algebra) for
  \begin{equation*}
    X
    :=
    \left(X_0:=\bigsqcup_{n\ge 1}\bC\bP^n\right)^+
    :=
    \quad
    \text{\emph{one-point compactification} \cite[Problem 19A]{wil_top} of $X_0$}
  \end{equation*}
  with fiber $\bC$ at the distinguished point $\infty\in X$ and fibers $M_2(\bC)$ over $X_0$. Because the associated $M_2$-bundle over $X$ is by construction not \emph{of finite type} (i.e. \cite[Definition 3.5.7]{hus_fib} trivializable over a finite open cover of $X_0$), $A$ cannot be topologically finitely-generated as a $C(X)$-module by \cite[Theorem 1.1]{gog_top-fg} (or \cite[Theorem A]{MR4936343}) and \cite[Proposition 3.7]{bg_cx-exp}.
\end{example}

\begin{remarks}\label{res:inf.g.ind}
  \begin{enumerate}[(1),wide]
  \item\label{item:res:inf.g.ind:yns} \Cref{ex:not.cent.fin} is one instance of the following general setup.
    \begin{itemize}[wide]
    \item Consider compact Hausdorff spaces $Y_n$ respectively equipped with (complex) vector $d$-bundles $\cE_n$; we abuse notation and conflate these \cite[Remark 15.1.2]{hjjm_bdle} with their corresponding principal $U(d)$-bundles. 

    \item Assume the attached set
      \begin{equation*}
        \left\{\mathrm{ind}_{U(d)}\left(\cE_n\right)\right\}_n
        \subseteq
        \bZ_{\ge 0}
      \end{equation*}
      of \emph{$U(d)$-indices} \cite[Definition 6.2.3]{mat_bu} is unbounded: there is no finite upper bound on the cardinality of an open cover of $Y_n$ that will trivialize $\cE_n$.

    \item The $\cE_n$ glue to a vector $d$-bundle over $X_0:=\bigsqcup_n Y_n$.

    \item Form the bundle $\cF:=\cE\oplus \left(X_0\times \bC\right)$ (i.e. add a trivial 1-dimensional summand to $\cE\xrightarrowdbl{} X_0$).

    \item Construct the corresponding endomorphism bundle $\cF\otimes \cF^*$ (a $(d+1)\times (d+1)$-matrix bundle over $X_0$).

    \item And finally, take for the $C^*$-algebra $A$ (supposed to play the same role as in \Cref{ex:not.cent.fin}) the \emph{unitization} \cite[\S II.1.2.1]{blk} $\Gamma_0\left(\cF\otimes \cF^*\right)^+$, the `0' subscript indicating sections on $X_0$ vanishing at $\infty$. 
    \end{itemize}

  \item\label{item:res:inf.g.ind:inf.dim} For path-connected (compact Hausdorff) $Y_n$ in \Cref{item:res:inf.g.ind:yns} above the unboundedness
    \begin{equation*}
      \sup_n \dim Y_n
      ,\quad
      \dim:=\text{\emph{covering dimension} \cite[Definition 1.6.7]{eng_dim}}
    \end{equation*}
    is an essential feature of this family of examples: per \cite[Proposition 2.1]{zbMATH03635101} (or as an immediate consequence of \cite[Lemma 2.4]{zbMATH03224293}, say) for any \emph{paracompact} \cite[\S 5.1]{eng_top_1989} path-connected space $Y$ there is an open cover
    \begin{equation*}
      Y=\bigcup_{i=0}^{\dim Y}U_i
      ,\quad
      U_i\text{ contractible in }Y.
    \end{equation*}
    The restriction of a bundle on $Y$ will be trivializable \cite[Theorem 14.3.3]{td_alg-top} over every $U_i$, rendering \Cref{item:res:inf.g.ind:yns} inoperative. 
  \end{enumerate}
\end{remarks}

Returning to $C_{\theta}$: while (mostly) not center-finite, they nevertheless exhibit less pathology in that regard than \Cref{ex:not.cent.fin} and the like. 

\begin{theorem}\label{th:ct.tfg.not.fg}
  Let $\theta\in M_n(\bQ)$ be a skew-symmetric matrix for some $n\in \bZ_{\ge 2}$. 
  \begin{enumerate}[(1),wide]
  \item\label{item:th:ct.tfg.not.fg:non.fg} The quantum sphere $C^*$-algebra $C_{\theta}$ is center-finite precisely when it is commutative, i.e. $\theta\in M_n(\bZ)$.

  \item\label{item:th:ct.tfg.not.fg:top.fg} On the other hand, $C_{\theta}$ is always topologically center-finite. 
  \end{enumerate}
\end{theorem}
\begin{proof}  
  \begin{enumerate}[label={},wide]
  \item\textbf{\Cref{item:th:ct.tfg.not.fg:non.fg}} \Cref{th:s2n1thetapi} makes $C_{\theta}$ into a subhomogeneous $Z_{\theta}$-algebra, which can thus \cite[Theorem A]{2409.03531v1} be regarded as a $Z_{\theta}$-\emph{Hilbert module} \cite[Definition 15.1.5]{wo}. Finite generation would imply \cite[Corollary 15.4.8]{wo} that $C_{\theta}$ is also \emph{projective} \cite[Definition 5.3.1]{agpr_pi}.

    The Hilbert-module-to-Hilbert-bundle correspondence of \cite[Scholium 6.7]{hk_shv-bdl} and Swan's celebrated \cite[Theorem 1.6.3]{ros_algk} (originally \cite[Theorems 1 and 2]{zbMATH03179258}) then imply that the \emph{(F) Hilbert bundle} \cite[p.9]{dg_ban-bdl} over $X:=\Max(Z_{\theta})$ with fibers
    \begin{equation*}
      X
      \ni p
      \xmapsto{\quad}
      C_{\theta}|_p
    \end{equation*}
    is locally trivial, so (the sphere being connected) of constant rank. This, in turn, is equivalent (\Cref{cor:azumaya.rare}) to the commutativity of $C_{\theta}$. 
    
  \item\textbf{\Cref{item:th:ct.tfg.not.fg:top.fg}} Once more regarding $C_{\theta}$, via \Cref{th:s2n1thetapi}\Cref{item:th:s2n1thetapi:s2n1-fincent}, as the section-space $\Gamma(\cE)$ of a subhomogeneous (F) Banach bundle $\cE\xrightarrowdbl{}\bS^{2n-1}\cong \Max(Z_{\theta})$, observe that the \emph{strata}
    \begin{equation*}
      X_d
      :=
      \left\{p\in X\ :\ \dim\text{fiber }\cE_p=\dim C_{\theta}/C_{\theta}\cdot p=d\right\}
    \end{equation*}
    are all members of the \emph{set ring} \cite[I, Definition 1.2.13]{bog_meas-1-2}  generated by the spaces
    \begin{equation*}
      \Max\left(Z_{\theta}|_{\bS^{n-1}_{+,F}}\right)
      \subseteq
      \Max\left(Z_{\theta}\right)
      =
      X
      ,\quad
      F\in 2^{[n]}.
    \end{equation*}
    They are all in particular finite unions of path-connected paracompact spaces of finite (covering) dimension, so \cite[Proposition 2.1]{zbMATH03635101} admit respective covers
    \begin{equation*}
      X_d
      =
      \bigcup_{j=0}^{N}U_j
      ,\quad
      U_j=\overset{\circ}{U}_j\text{ contractible in }\bS^{2n-1}_d
      \quad\left(\text{some }N\in \bZ_{\ge 0}\right).
    \end{equation*}
    The restrictions $\cE|_{X_d}$ are thus all trivializable by finite covers, hence topological finite generation via \cite[Theorem 1.1]{gog_top-fg}.
  \end{enumerate}
\end{proof}


\addcontentsline{toc}{section}{References}

\begin{thebibliography}{10}

\bibitem{agpr_pi}
Eli Aljadeff, Antonio Giambruno, Claudio Procesi, and Amitai Regev.
\newblock {\em Rings with polynomial identities and finite dimensional
  representations of algebras}, volume~66 of {\em American Mathematical Society
  Colloquium Publications}.
\newblock American Mathematical Society, Providence, RI, 2020.

\bibitem{ag_sepgrph}
P.~Ara and K.~R. Goodearl.
\newblock C{{\(^*\)}}-algebras of separated graphs.
\newblock {\em J. Funct. Anal.}, 261(9):2540--2568, 2011.

\bibitem{zbMATH03351818}
M.~Artin.
\newblock On {Azumaya} algebras and finite dimensional representations of
  rings.
\newblock {\em J. Algebra}, 11:532--563, 1969.

\bibitem{at_k}
M.~Atiyah.
\newblock {\em {$K$}-theory}.
\newblock Lecture notes by D. W. Anderson. W. A. Benjamin, Inc., New
  York-Amsterdam, 1967.

\bibitem{am_comm}
Michael~F. Atiyah and I.~G. Macdonald.
\newblock Introduction to commutative algebra.
\newblock Reading, {Mass}.-{Menlo} {Park}, {Calif}.-{London}-{Don} {Mills},
  {Ont}.: {Addison}-{Wesley} {Publishing} {Company} (1969)., 1969.

\bibitem{blk}
B.~Blackadar.
\newblock {\em Operator algebras}, volume 122 of {\em Encyclopaedia of
  Mathematical Sciences}.
\newblock Springer-Verlag, Berlin, 2006.
\newblock Theory of $C^*$-algebras and von Neumann algebras, Operator Algebras
  and Non-commutative Geometry, III.

\bibitem{bg_cx-exp}
Etienne Blanchard and Ilja Gogi\'{c}.
\newblock On unital {$C(X)$}-algebras and {$C(X)$}-valued conditional
  expectations of finite index.
\newblock {\em Linear Multilinear Algebra}, 64(12):2406--2418, 2016.

\bibitem{bog_meas-1-2}
V.~I. Bogachev.
\newblock {\em Measure theory. {V}ol. {I}, {II}}.
\newblock Springer-Verlag, Berlin, 2007.

\bibitem{bourb_alg-04-07}
Nicolas Bourbaki.
\newblock {\em Elements of mathematics. {Algebra} {II}. {Chapters} 4--7.
  {Transl}. from the {French} by {P}. {M}. {Cohn} and {J}. {Howie}.}
\newblock Berlin: Springer, reprint of the 1990 {English} translation edition,
  2003.

\bibitem{bourb_alg-09}
Nicolas Bourbaki.
\newblock {\em {\'E}l{\'e}ments de math{\'e}matique. {Alg{\`e}bre}. {Chapitre}
  9}.
\newblock Berlin: Springer, reprint of the 1959 original edition, 2007.

\bibitem{bg_algqg}
Ken~A. Brown and Ken~R. Goodearl.
\newblock {\em Lectures on algebraic quantum groups}.
\newblock Advanced Courses in Mathematics. CRM Barcelona. Birkh\"{a}user
  Verlag, Basel, 2002.

\bibitem{bursz_surv}
Henrique Bursztyn.
\newblock A survey on {M}orita equivalence of quantum tori, 1999.
\newblock available at
  \url{https://math.berkeley.edu/~alanw/242papers99/bursztyn} (accessed
  2023-05-31).

\bibitem{2409.03531v1}
Alexandru Chirvasitu.
\newblock Non-commutative branched covers and bundle unitarizability, 2024.
\newblock \url{http://arxiv.org/abs/2409.03531v1}.

\bibitem{MR4936343}
Alexandru Chirvasitu.
\newblock Small {B}anach bundles and modules.
\newblock {\em Proc. Amer. Math. Soc.}, 153(9):3907--3920, 2025.

\bibitem{cp_cont-equiv_xv3}
Alexandru Chirvasitu and Benjamin Passer.
\newblock Continuity and equivariant dimension, 2024.
\newblock \url{http://arxiv.org/abs/2403.06767v3}.

\bibitem{dcp_qg}
C.~De~Concini and C.~Procesi.
\newblock Quantum groups.
\newblock In {\em D-modules, representation theory, and quantum groups.
  Lectures given at the 2nd session of the Centro Internazionale Matematico
  Estivo (C.I.M.E.) held in Venezia, Italy, June 12-20, 1992}, pages 31--140.
  Berlin: Springer-Verlag, 1993.

\bibitem{dixc}
Jacques Dixmier.
\newblock {\em {$C\sp*$}-algebras}.
\newblock North-Holland Mathematical Library, Vol. 15. North-Holland Publishing
  Co., Amsterdam-New York-Oxford, 1977.
\newblock Translated from the French by Francis Jellett.

\bibitem{df_pi}
Vesselin Drensky and Edward Formanek.
\newblock {\em Polynomial identity rings.}
\newblock Adv. Courses in Math. -- CRM Barc. Basel: Birkh{\"a}user, 2004.

\bibitem{dg_ban-bdl}
M.~J. Dupr{\'e} and R.~M. Gillette.
\newblock {\em Banach bundles, {Banach} modules and automorphisms of
  {{\(C^*\)}}-algebras}, volume~92 of {\em Res. Notes Math., San Franc.}
\newblock Pitman Publishing, London, 1983.

\bibitem{dupr_clsf-cast-bdl}
Maurice~J. Dupre.
\newblock {\em The classification and structure of {C}*-algebra bundles},
  volume 222 of {\em Mem. Am. Math. Soc.}
\newblock Providence, RI: American Mathematical Society (AMS), 1979.

\bibitem{el_kproj}
G.~A. Elliott.
\newblock On the {K}-theory of the {{\(C^*\)}}-algebra generated by a
  projective representation of a torsion-free discrete abelian group.
\newblock Operator algebras and group representations, {Proc}. int. {Conf}.,
  {Neptun}/{Rom}. 1980, {Vol}. {I}, {Monogr}. {Stud}. {Math}. 17, 157-184
  (1984)., 1984.

\bibitem{el_morita-def}
George~A. Elliott and Hanfeng Li.
\newblock Strong {M}orita equivalence of higher-dimensional noncommutative
  tori. {II}.
\newblock {\em Math. Ann.}, 341(4):825--844, 2008.

\bibitem{eng_dim}
Ryszard Engelking.
\newblock {\em Dimension theory}.
\newblock North-Holland Publishing Co., Amsterdam-Oxford-New York; PWN---Polish
  Scientific Publishers, Warsaw, 1978.
\newblock Translated from the Polish and revised by the author, North-Holland
  Mathematical Library, 19.

\bibitem{engel_dim}
Ryszard Engelking.
\newblock {\em Dimension theory. {A} revised and enlarged translation of
  ''{Teoria} wymiaru'', {Warszawa} 1977, by the author}, volume~19 of {\em
  North-Holland Math. Libr.}
\newblock Elsevier (North-Holland), Amsterdam, 1978.

\bibitem{eng_top_1989}
Ryszard Engelking.
\newblock {\em General topology.}, volume~6 of {\em Sigma Ser. Pure Math.}
\newblock Berlin: Heldermann Verlag, rev. and compl. ed. edition, 1989.

\bibitem{fd_bdl-1}
J.~M.~G. Fell and R.~S. Doran.
\newblock {\em Representations of *-algebras, locally compact groups, and
  {Banach} *- algebraic bundles. {Vol}. 1: {Basic} representation theory of
  groups and algebras}, volume 125 of {\em Pure Appl. Math., Academic Press}.
\newblock Boston, MA etc.: Academic Press, Inc., 1988.

\bibitem{fk_fin-ind}
Michael Frank and Eberhard Kirchberg.
\newblock On conditional expectations of finite index.
\newblock {\em J. Oper. Theory}, 40(1):87--111, 1998.

\bibitem{zbMATH04101395}
Tatiana Gateva.
\newblock P.i. degree of tensor products of {PI}-algebras.
\newblock {\em J. Algebra}, 123(1):64--73, 1989.

\bibitem{gj_rings}
Leonard Gillman and Meyer Jerison.
\newblock {\em Rings of continuous functions}.
\newblock Graduate Texts in Mathematics, No. 43. Springer-Verlag, New
  York-Heidelberg, 1976.
\newblock Reprint of the 1960 edition.

\bibitem{gog_top-fg}
Ilja Gogi{\'c}.
\newblock Topologically finitely generated {Hilbert} {{\(C(X)\)}}-modules.
\newblock {\em J. Math. Anal. Appl.}, 395(2):559--568, 2012.

\bibitem{gvf_ncg}
Jos\'{e}~M. Gracia-Bond\'{\i}a, Joseph~C. V\'{a}rilly, and H\'{e}ctor Figueroa.
\newblock {\em Elements of noncommutative geometry}.
\newblock Birkh\"{a}user Advanced Texts: Basler Lehrb\"{u}cher. [Birkh\"{a}user
  Advanced Texts: Basel Textbooks]. Birkh\"{a}user Boston, Inc., Boston, MA,
  2001.

\bibitem{hks_erg}
Raphael H{\o}egh-Krohn and Tor Skjelbred.
\newblock Classification of {$C^{\ast} $}-algebras admitting ergodic actions of
  the two-dimensional torus.
\newblock {\em J. Reine Angew. Math.}, 328:1--8, 1981.

\bibitem{hm4}
Karl~H. Hofmann and Sidney~A. Morris.
\newblock {\em The structure of compact groups---a primer for the student---a
  handbook for the expert}, volume~25 of {\em De Gruyter Studies in
  Mathematics}.
\newblock De Gruyter, Berlin, [2020] \copyright 2020.
\newblock Fourth edition [of 1646190].

\bibitem{hk_shv-bdl}
Karl~Heinrich Hofmann and Klaus Keimel.
\newblock Sheaf theoretical concepts in analysis: {Bundles} and sheaves of
  {Banach} spaces, {Banach} {C}({X})-modules.
\newblock Applications of sheaves, {Proc}. {Res}. {Symp}., {Durham} 1977,
  {Lect}. {Notes} {Math}. 753, 415-441 (1979)., 1979.

\bibitem{hjjm_bdle}
D.~Husem\"{o}ller, M.~Joachim, B.~Jur\v{c}o, and M.~Schottenloher.
\newblock {\em Basic bundle theory and {$K$}-cohomology invariants}, volume 726
  of {\em Lecture Notes in Physics}.
\newblock Springer, Berlin, 2008.
\newblock With contributions by Siegfried Echterhoff, Stefan Fredenhagen and
  Bernhard Kr\"{o}tz.

\bibitem{hus_fib}
Dale Husemoller.
\newblock {\em Fibre bundles}, volume~20 of {\em Graduate Texts in
  Mathematics}.
\newblock Springer-Verlag, New York, third edition, 1994.

\bibitem{zbMATH03635101}
I.~M. James.
\newblock On category, in the sense of {Lusternik}-{Schnirelmann}.
\newblock {\em Topology}, 17:331--348, 1978.

\bibitem{zbMATH04056334}
G.~G. Kasparov.
\newblock Equivariant {KK}-theory and the {Novikov} conjecture.
\newblock {\em Invent. Math.}, 91(1):147--201, 1988.

\bibitem{khal_basic}
Masoud Khalkhali.
\newblock {\em Basic noncommutative geometry}.
\newblock EMS Ser. Lect. Math. Z{\"u}rich: European Mathematical Society (EMS),
  2nd updated ed. edition, 2013.

\bibitem{kob_cplx}
Shoshichi Kobayashi.
\newblock {\em Differential geometry of complex vector bundles}.
\newblock Princeton Legacy Library. Princeton University Press, Princeton, NJ,
  [2014].
\newblock Reprint of the 1987 edition [ MR0909698].

\bibitem{koch_bord}
S.~O. Kochman.
\newblock {\em Bordism, stable homotopy and {A}dams spectral sequences},
  volume~7 of {\em Fields Institute Monographs}.
\newblock American Mathematical Society, Providence, RI, 1996.

\bibitem{mcl_work}
Saunders Mac~Lane.
\newblock {\em Categories for the working mathematician.}, volume~5 of {\em
  Grad. Texts Math.}
\newblock New York, NY: Springer, 2nd ed edition, 1998.

\bibitem{mat_bu}
Ji{\v{r}}{\'{\i}} Matou{\v{s}}ek.
\newblock {\em Using the {Borsuk}-{Ulam} theorem. {Lectures} on topological
  methods in combinatorics and geometry. {Written} in cooperation with {Anders}
  {Bj{\"o}rner} and {G{\"u}nter} {M}. {Ziegler}}.
\newblock Universitext. Berlin: Springer, 2nd corrected printing edition, 2008.

\bibitem{mcc_kext}
Kevin McClanahan.
\newblock {{\(K\)}}-theory and ext-theory for rectangular unitary
  {{\(C^*\)}}-algebras.
\newblock {\em Rocky Mt. J. Math.}, 23(3):1063--1080, 1993.

\bibitem{mr}
J.~C. McConnell and J.~C. Robson.
\newblock {\em Noncommutative {N}oetherian rings}, volume~30 of {\em Graduate
  Studies in Mathematics}.
\newblock American Mathematical Society, Providence, RI, revised edition, 2001.
\newblock With the cooperation of L. W. Small.

\bibitem{mitch_notes-2011}
Stephen~A. Mitchell.
\newblock Notes on principal bundles and classifying spaces, 2011.
\newblock available at
  \url{https://sites.math.washington.edu//~mitchell/Notes/prin.pdf} (accessed
  2023-05-31).

\bibitem{no_sph}
T.~Natsume and C.~L. Olsen.
\newblock Toeplitz operators on noncommutative spheres and an index theorem.
\newblock {\em Indiana Univ. Math. J.}, 46(4):1055--1112, 1997.

\bibitem{oda_vb-ab}
T.~Oda.
\newblock Vector bundles on {Abelian} surfaces.
\newblock {\em Invent. Math.}, 13:247--260, 1971.

\bibitem{zbMATH03224293}
R.~S. Palais.
\newblock Homotopy theory of infinite dimensional manifolds.
\newblock {\em Topology}, 5:1--16, 1966.

\bibitem{pt_brnch}
A.~A. Pavlov and E.~V. Troitskii.
\newblock Quantization of branched coverings.
\newblock {\em Russ. J. Math. Phys.}, 18(3):338--352, 2011.

\bibitem{zbMATH05172034}
N.~Christopher Phillips.
\newblock Recursive subhomogeneous algebras.
\newblock {\em Trans. Am. Math. Soc.}, 359(10):4595--4623, 2007.

\bibitem{pis_os}
Gilles Pisier.
\newblock {\em Introduction to operator space theory}, volume 294 of {\em Lond.
  Math. Soc. Lect. Note Ser.}
\newblock Cambridge: Cambridge University Press, 2003.

\bibitem{proc_pi}
Claudio Procesi.
\newblock {\em Rings with polynomial identities}, volume~17 of {\em Pure and
  Applied Mathematics}.
\newblock Marcel Dekker, Inc., New York, 1973.

\bibitem{rief_irrat}
Marc~A. Rieffel.
\newblock {$C^{\ast} $}-algebras associated with irrational rotations.
\newblock {\em Pacific J. Math.}, 93(2):415--429, 1981.

\bibitem{rief_canc}
Marc~A. Rieffel.
\newblock The cancellation theorem for projective modules over irrational
  rotation {$C^{\ast} $}-algebras.
\newblock {\em Proc. London Math. Soc. (3)}, 47(2):285--302, 1983.

\bibitem{rief_case}
Marc~A. Rieffel.
\newblock Noncommutative tori---a case study of noncommutative differentiable
  manifolds.
\newblock In {\em Geometric and topological invariants of elliptic operators
  ({B}runswick, {ME}, 1988)}, volume 105 of {\em Contemp. Math.}, pages
  191--211. Amer. Math. Soc., Providence, RI, 1990.

\bibitem{riehl_catht}
Emily Riehl.
\newblock {\em Categorical homotopy theory}, volume~24 of {\em New Math.
  Monogr.}
\newblock Cambridge: Cambridge University Press, 2014.

\bibitem{ros_algk}
Jonathan Rosenberg.
\newblock {\em Algebraic {K}-theory and its applications}, volume 147 of {\em
  Grad. Texts Math.}
\newblock New York, NY: Springer-Verlag, 1994.

\bibitem{rs_dgmp-2}
Gerd Rudolph and Matthias Schmidt.
\newblock {\em Differential geometry and mathematical physics. {II}: {Fibre}
  bundles, topology and gauge fields}.
\newblock Theor. Math. Phys. (Cham). Dordrecht: Springer, 2017.

\bibitem{MR1886684}
Martin Scharlemann.
\newblock Heegaard splittings of compact 3-manifolds.
\newblock In {\em Handbook of geometric topology}, pages 921--953.
  North-Holland, Amsterdam, 2002.

\bibitem{spa_at}
Edwin~H. Spanier.
\newblock {\em Algebraic topology}.
\newblock Berlin: Springer-Verlag, 1995.

\bibitem{zbMATH03179258}
R.~G. Swan.
\newblock Vector bundles and projective modules.
\newblock {\em Trans. Am. Math. Soc.}, 105:264--277, 1962.

\bibitem{tak1}
M.~Takesaki.
\newblock {\em Theory of operator algebras. {I}}, volume 124 of {\em
  Encyclopaedia of Mathematical Sciences}.
\newblock Springer-Verlag, Berlin, 2002.
\newblock Reprint of the first (1979) edition, Operator Algebras and
  Non-commutative Geometry, 5.

\bibitem{thad_introtop}
Michael Thaddeus.
\newblock An introduction to the topology of the moduli space of stable bundles
  on a {R}iemann surface.
\newblock In {\em Geometry and physics ({A}arhus, 1995)}, volume 184 of {\em
  Lecture Notes in Pure and Appl. Math.}, pages 71--99. Dekker, New York, 1997.

\bibitem{zbMATH05905742}
Akaki Tikaradze.
\newblock On the {Azumaya} locus of almost commutative algebras.
\newblock {\em Proc. Am. Math. Soc.}, 139(6):1955--1960, 2011.

\bibitem{td_alg-top}
Tammo tom Dieck.
\newblock {\em Algebraic topology}.
\newblock EMS Textb. Math. Z{\"u}rich: European Mathematical Society (EMS),
  2008.

\bibitem{wo}
N.~E. Wegge-Olsen.
\newblock {\em {$K$}-theory and {$C^*$}-algebras}.
\newblock Oxford Science Publications. The Clarendon Press, Oxford University
  Press, New York, 1993.
\newblock A friendly approach.

\bibitem{wil_top}
Stephen Willard.
\newblock {\em General topology}.
\newblock Dover Publications, Inc., Mineola, NY, 2004.
\newblock Reprint of the 1970 original [Addison-Wesley, Reading, MA;
  MR0264581].

\end{thebibliography}

\def\polhk#1{\setbox0=\hbox{#1}{\ooalign{\hidewidth
  \lower1.5ex\hbox{`}\hidewidth\crcr\unhbox0}}}
  \def\polhk#1{\setbox0=\hbox{#1}{\ooalign{\hidewidth
  \lower1.5ex\hbox{`}\hidewidth\crcr\unhbox0}}}
  \def\polhk#1{\setbox0=\hbox{#1}{\ooalign{\hidewidth
  \lower1.5ex\hbox{`}\hidewidth\crcr\unhbox0}}}
  \def\polhk#1{\setbox0=\hbox{#1}{\ooalign{\hidewidth
  \lower1.5ex\hbox{`}\hidewidth\crcr\unhbox0}}}
  \def\polhk#1{\setbox0=\hbox{#1}{\ooalign{\hidewidth
  \lower1.5ex\hbox{`}\hidewidth\crcr\unhbox0}}}
  \def\polhk#1{\setbox0=\hbox{#1}{\ooalign{\hidewidth
  \lower1.5ex\hbox{`}\hidewidth\crcr\unhbox0}}}

\Addresses

\end{document}